\documentclass[a4paper,10pt]{article}

\usepackage[utf8]{inputenc}
\usepackage[T1]{fontenc}
\usepackage{amsmath, amsfonts, amssymb, amsthm}
\usepackage{mathtools}
\usepackage{enumerate}
\usepackage{authblk}

\newcommand{\R}{\mathbb{R}}
\newcommand{\N}{\mathbb{N}}
\newcommand{\B}{\mathcal{B}}    
\newcommand{\K}{\mathcal{K}}	
\newcommand{\M}{\mathcal{M}}	
\newcommand{\Ms}{M_{s_o}}
\newcommand{\Prob}{\mathbb{P}}  
\newcommand{\F}{\mathcal{F}}
\newcommand{\1}{\operatorname{\mathbf{1}}}	

\newcommand{\Haus}{\mathcal{H}}    

\DeclareMathOperator{\supp}{supp}			
\DeclareMathOperator{\aff}{aff}				

\newcommand{\sH[1]}[k]{\mathcal{H}^n_{#1}}				
\newcommand{\abs[1]}{\lvert #1 \rvert}					
\newcommand{\norm[1]}{\lVert #1 \rVert}   				
\newcommand{\IP}[2]{\langle #1, #2 \rangle}				

\newcommand{\nto}[1]{\norm[#1]_2}						
\newcommand{\IPto}[2]{\langle #1, #2 \rangle_2}			

\newcommand{\BL}{BL(S^{n-1})} 				
\newcommand{\nBL[1]}{\norm[#1]_{BL}} 	    
\newcommand{\nsup[1]}{\norm[#1]_{\infty}}   
\newcommand{\nL[1]}{\norm[#1]_L}			

\newcommand{\dudley}[2]{d_D(#1, #2)}		

\newcommand{\Poly[1]}[m_{s_o}]{\mathcal{P}_{#1}}	
\newcommand{\mK}{\hat{K}_{s_o}}    					
\newcommand{\Ks}{\mathbb{K}_{s_o}(\epsilon_{s_o})}  

\newcommand{\proj[1]}[k]{\operatorname{\Pi_{n#1}}}

\newcommand{\T[1]}[s]{\Phi_{n-1}^{#1}(K)}		 
\newcommand{\Tnul[1]}[s]{\Phi_{n-1}^{#1}(K_0)}   

\newcommand{\HV[1]}[s_o]{\psi_{n-1}^{#1}}		 

\newtheorem{proposition}{Proposition}[section]
\newtheorem{theorem}[proposition]{Theorem}
\newtheorem{corollary}[proposition]{Corollary}
\newtheorem{lemma}[proposition]{Lemma}

\theoremstyle{definition}
\newtheorem{remark}[proposition]{Remark}

\usepackage{xspace} 
\makeatletter
\renewenvironment{abstract}{%
  \small%
  \providecommand\keywords{%
    \par\medskip\noindent\textit{Keywords:}\xspace}%
  \begin{center}%
    \bfseries \abstractname\vspace{-.5em}\vspace{\z@}%
  \end{center}%
  \quote%
}{\endquote}
\makeatother

\title{Reconstruction of $n$-dimensional convex bodies from surface tensors}
\author{Astrid Kousholt\footnote{kousholt@math.au.dk}}
\date{}
\affil{Department of Mathematics, Aarhus University, Denmark}

\begin{document}
\maketitle

\vspace*{-2em}

\begin{abstract}
In this paper, we derive uniqueness and stability results for surface tensors. Further, we develop two algorithms that reconstruct shape of $n$-dimensional convex bodies. One algorithm requires knowledge of a finite number of surface tensors, whereas the other algorithm is based on noisy measurements of a finite number of harmonic intrinsic volumes. The derived stability results ensure consistency of the two algorithms. Examples that illustrate the feasibility of the algorithms are presented. 

\\

\end{abstract}

\section{Introduction}
Recently, Minkowski tensors have succesfully been used as shape descriptors of spatial structures in materials science, see, e.g., \cite{Beisbart2002,SM10, Schroder-Turk2013}. Surface tensors are translation invariant Minkowski tensors derived from surface area measures, and the shape of a convex body $K$ with nonempty interior in $\R^n$ is uniquely determined by the surface tensors of $K$. In this context, the shape of $K$ is defined as the equivalence class of all translations of $K$. 

In \cite{Kousholt2016}, Kousholt and Kiderlen develop reconstruction algorithms that approximate the shape of convex bodies in $\R^2$ from a finite number of surface tensors. Kousholt and Kiderlen describe two algorithms. One algorithm requires knowledge of exact surface tensors and one allows for noisy measurements of surface tensors. For the latter algorithm, it is argued that it is preferable to use harmonic intrinsic volumes instead of surface tensors evaluated at the standard basis. The purpose of this paper is threefold. Firstly, the reconstruction algorithms in \cite{Kousholt2016} are generalized to an $n$-dimensional setting. Secondly, stability and uniqueness results for surface tensors are established, and the stability results are used to ensure consistency of the generalized algorithms. Thirdly, we illustrate the feasibility of the reconstruction algorithms by examples. The generalizations of the reconstruction algorithms are developed along the same lines as the algorithms for convex bodies in $\R^2$. However, there are several non-trivial obstacles on the way. In particular, essentially different stability results are needed to ensure consistency.

The input of the first generalized algorithm is exact surface tensors up to a certain rank of an unknown convex body in $\R^n$. The output is a polytope with surface tensors identical to the given surface tensors of the unknown convex body. The input of the second generalized algorithm is measurements of harmonic intrinsic volumes of an unknown convex body in $\R^n$, and the output is a polytope with harmonic intrinsic volumes that fit the given measurements in a least squares sense. When $n \geq 3$, a convex body that fits the input measurements of harmonic intrinsic volumes may not exist, and in this case, the algorithm based on harmonic intrinsic volumes does not have an output. However, this situation only occurs when the measurements are too noisy, see Lemma~\ref{Lemma_D_bound}.

The consistency of the algorithms described in \cite{Kousholt2016} is established using the stability result \cite[Thm. 4.8]{Kousholt2016} for harmonic intrinsic volumes derived from the first order area measure. This result can be applied as the first order area measure and the surface area measure coincide for $n=2$. However, for $n \geq 3$, the stability result is not applicable. Therefore, we establish stability results for surface tensors and for harmonic intrinsic volumes derived from surface area measures. More precisely, first we derive an upper bound of the Dudley distance between surface area measures of two convex bodies. This bound is small, when $s$ is large and the distance between the harmonic intrinsic volumes up to degree $s$ of the convex bodies is small (Theorem~\ref{Thm_noise}). From this result and a known connection between the Dudley distance and the translative Hausdorff distance, we obtain that the translative Hausdorff distance between convex bodies with identical surface tensors up to rank $s$ becomes small, when $s$ is large (Corollary~\ref{Cor_bound_dudley}). The stability result for surface tensors and the fact that the rank $2$ surface tensor of a convex body $K$ determines the radii of a ball containing $K$ and a ball contained in $K$ (Lemma~\ref{spheres}) ensure consistency of the generalized reconstruction algorithm based on exact surface tensors (Theorem~\ref{thm_conv}). The consistency of the reconstruction algorithm based on measurements of harmonic intrinsic volumes are ensured by the stability result for harmonic intrinsic volumes under certain assumptions on the variance of the noise variables (Theorems~\ref{thm_cons_as} and \ref{thm_cons_prob}). 

The described algorithms and stability results show that a finite number of surface tensors can be used to approximate the shape of a convex body, but in general, all surface tensors are required to uniquely determine the shape of a convex body. However, there are convex bodies where a finite number of surface tensors contain full information about the shapes of the convex bodies. More precisely, in \cite{Kousholt2016}, it is shown that the shape of a convex body in $\R^n$ with nonempty interior is uniquely determined by a finite number of surface tensors only if the convex body is a polytope. We complement this result by showing that the shape of a polytope with $m$ facets is uniquely determined by the surface tensors up to rank $m - n +2$. This result is optimal in the sense that for each $m \geq n+1$ there is a polytope $P$ with $m$ facets and a convex body $K$ that is not a polytope, such that $P$ and $K$ have identical surface tensors up to rank $m-n+1$. This implies that the rank $m-n+2$ cannot be reduced. An earlier and weaker result in this direction is \cite[Thm. 4.3]{Kousholt2016} stating that the shape of a polytope with $m$ facets is determined by the surface tensors up to rank $2m$.

The paper is organized as follows. General notation, surface tensors and harmonic intrinsic volumes are introduced in Section~\ref{Sec_prelim}. The uniqueness results are derived in Section~\ref{Sec_Uniqueness} and are followed by the stability results in Section~\ref{Sec_stab}. The two reconstruction algorithms are described in Sections~\ref{Sec_Recon} and \ref{Sec_ReconHVol}.

\section{Notation and preliminaries}\label{Sec_prelim}
We work in the $n$-dimensional Euclidean vector space $\R^n$, $n \geq 2$ with standard inner product $\IP{\cdot}{\cdot}$ and induced norm $\norm[\cdot]$. The unit sphere in $\R^n$ is denoted $S^{n-1}$, and the surface area and volume of the unit ball $B^n$ in $\R^n$ is denoted $\omega_n$ and $\kappa_n$, respectively.

In the following, we give a brief introduction to the concepts of convex bodies, surface area measures, surface tensors and harmonic intrinsic volumes. For further details, we refer to \cite{Schneider14} and \cite{Kousholt2016}. We let $\K^n$ denote the set of convex bodies (convex, compact and nonempty sets) in $\R^n$, and let $\K^n_n$ be the set of convex bodies with nonempty interior. Further, $\K^n(R)$ is the set of convex bodies contained in a ball of radius $R> 0$, and likewise, $\K^n(r,R)$ is the set of convex bodies that contain a ball of radius $r > 0$ and are contained in a concentric ball of radius $R > r$. The set of convex bodies $\K^n$ is equipped with the Hausdorff metric $\delta$. The Hausdorff distance between two convex bodies can be expressed as the supremum norm of the difference of the support functions of the convex bodies, i.e. 
\begin{equation*}
\delta(K,L) = \nsup[h_K - h_L] = \sup_{u \in S^{n-1}}\abs[h_K(u) - h_L(u)]
\end{equation*}
for $K, L \in \K^n$.

In the present work, we call the equivalence class of translations of a convex body $K$ the shape of $K$. Hence, two convex bodies are of the same shape exactly if they are translates. As a measure of distance in shape, we use the translative Hausdorff distance,
\begin{equation*}
\delta^t(K,L)=\inf_{x \in \R^n} \delta(K, L + x)
\end{equation*}
for $K, L \in \K^n$. The translative Hausdorff distance is a metric on the set of shapes of convex bodies, see \cite[p. 165]{Gardner06}.

For a convex body $K \in \K^n_n$, the surface area measure $S_{n-1}(K, \cdot)$ of $K$ is defined as
$$S_{n-1}(K, \omega) = \Haus^{n-1}(\tau(K, \omega))$$
for a Borel set $\omega \subseteq S^{n-1}$, where $\Haus^{n-1}$ is the $(n-1)$-dimensional Hausdorff measure, and $\tau(K, \omega)$ is the set of boundary points of $K$ with an outer normal belonging to $\omega$. For a convex body $K \in \K^n \setminus \K^n_n$ there is a unit vector $u \in S^{n-1}$ and an $x \in \R^n$, such that $K$ is contained in the hyperplane $u^{\perp} + x$. The surface area measure of $K$ is defined as $$S_{n-1}(K, \cdot)=S(K)(\delta_u + \delta_{-u}),$$ where $S(K)$ is the surface area of $K$ and $\delta_v$ is the Dirac measure at $v \in S^{n-1}$. Notice that $S(K)=S_{n-1}(K, S^{n-1})$ for $K \in \K^n_n$, and $2S(K)=S^{n-1}(K, S^{n-1})$ for $K \in \K^n \setminus \K^n_n$.

The surface tensors of $K \in \K^n$ are the Minkowski tensors of $K$ derived from the surface area measure of $K$. Hence for $s \in \N_0$, \emph{the surface tensor} of $K$ of rank $s$ is given as
\begin{equation*}
\T[s]= \frac{1}{s! \, \omega_{s+1}}\int_{S^{n-1}} u^s \, S_{n-1}(K, du)
\end{equation*}
where $u^s \colon (\R^n)^s \to \R$ is the $s$-fold symmetric tensor product of $u \in S^{n-1}$ when $u$ is identified with the rank $1$ tensor $v \mapsto \IP{u}{v}$. Due to multilinearity, the surface tensor of rank $s$ can be identified with the array $\{\T[s](e_{i_1}, \dots, e_{i_s})\}_{i_1, \dots, i_s =1}^n$  of components of $\T[s]$, where $(e_1, \dots, e_n)$ is the standard basis of $\R^n$. Notice that the the components of $\T[s]$ are scaled versions of the moments of $S_{n-1}(K, \cdot)$, where the moments of order $s\in \N_0$ of a Borel measure $\mu$ on $S^{n-1}$ are given by
\begin{equation*}
\int_{S^{n-1}} u_1^{i_1} \cdots u_n^{i_n} \ \mu(du)
\end{equation*}
for $i_1, \dots, i_n \in \{0, \dots, s\}$ with $\sum_{j=1}^n i_j = s$.

By \cite[Remark 3.1]{Kousholt2016}, the surface tensors $\T[0], \dots, \T[s]$ of $K$ are uniquely determined by $\T[s-1]$ and $\T[s]$ for $s \geq 2$. More precisely, if $0 \leq s \leq s_o$ has same parity as $s_o$, say, then $\Phi_{n-1}^s$ can be calculated from $\Phi_{n-1}^{s_o}$ by taking the trace consecutively and multiplying with the constant
\begin{equation}\label{Constant_c}
c_{s,s_o}=\frac{s_o! \, \omega_{s_o}}{s! \, \omega_{s+1}}.
\end{equation}        
We let 
\begin{equation*}
m_s = \binom{s + n -2}{n-1} + \binom{s+n-1}{n-1}
\end{equation*}
be the number of \textit{different} components of $\T[s-1]$ and $\T[s]$, and we use the notation $\phi_{n-1}^s(K)$ for the $m_s$-dimensional vector of different components of the surface tensors of $K$ of rank $s-1$ and $s$.

To a convex body $K \in \K^n$, we further associate the harmonic intrinsic volumes that are the moments of $S_{n-1}(K, \cdot)$ with respect to an orthonormal sequence of spherical harmonics (for details on spherical harmonics, see \cite{Groemer1996}). More precisely, for $k \in \N_0$, let $\sH$ be the vector space of spherical harmonics of degree $k$ on $S^{n-1}$. The dimension of $\sH$ is denoted $N(n,k)$, and $\sum_{k=0}^s N(n,k) = m_s$. We let $H_{nk1}, \dots, H_{nkN(n,k)}$ be an orthonormal basis of $\sH$. Then, the \emph{harmonic intrinsic volumes} of $K$ of degree $k$ are given by
\begin{equation*}
\psi_{(n-1)kj}(K)=\int_{S^{n-1}} H_{nkj}(u) \, S_{n-1}(K, du)
\end{equation*}
for $j=1, \dots, N(n,k)$. For a convex body $K \in \K^n$, we let $\HV[s](K)$ be the $m_s$-dimensional vector of harmonic intrinsic volumes of $K$ up to degree $s$. The vector $\HV[s](K)$ only depends on $K$ through the surface area measure $S_{n-1}(K, \cdot)$ of $K$, and we can write $\HV[s](S_{n-1}(K, \cdot))=\HV[s](K)$. Likewise, for an arbitrary Borel measure $\mu$ on $S^{n-1}$, we write $\HV[s](\mu)$ for the vector of harmonic intrinsic volumes of $\mu$ up to order $s$, that is the vector of moments of $\mu$ up to order $s$ with respect to the given orthonormal basis of spherical harmonics. The harmonic intrinsic volumes and the surface tensors of a convex body $K$ are closely related as there is an invertible linear mapping $f \colon \R^{m_s} \to \R^{m_s}$ such that $f(\phi_{n-1}^s(K))=\HV[s](K)$.

\section{Uniqueness results}\label{Sec_Uniqueness}

The shape of a convex body is uniquely determined by a finite number of surface tensors only if the convex body is a polytope, see \cite[Cor. 4.2]{Kousholt2016}. Further, in \cite[Thm. 4.3]{Kousholt2016} it is shown that a polytope in $\R^n$ with nonempty interior and $m \geq n+1$ facets is uniquely determined up to translation in $\K^n$ by its surface tensors up to rank $2m$. In Theorem~\ref{Thm_uniqueness}, we replace $2m$ with $m -n +2$, and in addition, we show that the rank $m -n +2$ cannot be reduced.

We let $\M$ denote the cone of finite Borel measures on $S^{n-1}$. Further, we let $\Poly[m]$ be the set of convex polytopes in $\R^n$ with at most $m \geq n+1$ facets. The proof of Lemma~\ref{lemma_uniqueness} is an improved version of the proof of \cite[Thm. 4.3]{Kousholt2016}. 

\begin{lemma} \label{lemma_uniqueness}
Let $m \in \N$ and $\mu \in \M$ have finite support $\{u_1, \dots, u_m\} \subseteq S^{n-1}$. 
\begin{enumerate}[(i)]
\item\label{item01} The measure $\mu$ is uniquely determined in $\M$ by its moments up to order $m$.
\item\label{item02} If the affine hull $\aff \{u_1, \dots, u_m\}$ of $\supp \mu$ is $\R^n$, then $\mu$ is uniquely determined in $\M$ by its moments up to order $m-n+2$. 
\end{enumerate}
\end{lemma}

\begin{proof} We first prove $\eqref{item02}$. Since $\aff \{u_1, \dots, u_m\}=\R^n$, we have $m \geq n+1$ and the support of $\mu$ can be pared down to $n+1$ vectors, say $u_1, \dots, u_{n+1}$, such that $\aff\{u_1, \dots, u_{n+1}\}=\R^n$. For each $j=1, \dots, n+1$, the affine hull 
\begin{equation*}
A_j = \aff (\{u_1, \dots, u_{n+1}\} \setminus\{u_j\})
\end{equation*}
is a hyperplane in $\R^n$, so there is a $v^j \in S^{n-1}$ and $\beta_j \in \R$ such that
\begin{equation*}
A_j = \{ x \in \R^n \mid \IP{x}{v^j} = \beta_j\}. 
\end{equation*}
Now define the polynomial    
\begin{equation*}
 p(u)=\sum_{j=1}^{n+1} (\IP{u}{v^j}- \beta_j)^2(1- \IP{u}{u_{j}})(1- \IP{u}{u_{n+2}})\dots (1-\IP{u}{u_m})
\end{equation*}
for $u \in S^{n-1}$. The degree of $p$ is $m-n+2$, and $p(u_j)=0$ for $j=1, \dots, m$. Let $w \in S^{n-1}\setminus \{u_1, \dots, u_{m}\}$ and assume that $p(w)=0$. Then 
$w \in A_j$ for $j=1, \dots, n+1$, so in particular $w=\sum_{j=1}^{n} \gamma_j u_j$ where $\sum_{j=1}^{n} \gamma_j=1$.
We may assume that $\gamma_1 \neq 0$. Since $w \in A_1$, this implies that $u_1$ is an affine combination of $u_2, \dots, u_{n+1}$, so
\begin{equation*}
A_1 = \aff\{u_1, \dots, u_{n+1}\} = \R^n.
\end{equation*}
This is a contradiction, and we conclude that $p(w) > 0$.

Now let $\nu \in \M$ and assume that $\mu$ and $\nu$ have identical moments up to order $m-n+2$. Since the polynomial $p$ is of degree $m-n+2$, we obtain that
\begin{equation}\label{Eq1}
\int_{S^{n-1}} p(u) \, \nu (du) = \int_{S^{n-1}} p(u) \, \mu (du) = \sum_{j=1}^m \alpha_j \, p(u_j) = 0,
\end{equation}
where we have used that $\mu$ is of the form
\begin{equation*}
\mu = \sum_{j=1}^m \alpha_j \delta_{u_j}
\end{equation*}
for some $\alpha_1, \dots, \alpha_m > 0$.
Equation~\eqref{Eq1} yields that $p(u) = 0$ for $\nu$-almost all $u \in S^{n-1}$ as the polynomial $p$ is non-negative. Then, the continuity of $p$ implies that
\begin{equation*}
\supp \nu \subseteq \{u \in S^{n-1} \mid p(u) = 0\} = \{u_1, \dots, u_m\},
\end{equation*}
so $\nu$ is of the form
\begin{equation}\label{form2}
\nu = \sum_{j=1}^m \beta_j \delta_{u_j}
\end{equation}
with $\beta_j \geq 0$ for $j=1, \dots, m$.

For $i=1, \dots, n+1$, define the polynomial
\begin{equation*}
p_i(u)=(\IP{u}{v^i} - \beta_i)^2(1-\IP{u}{u_{n+2}}) \dots (1-\IP{u}{u_m})
\end{equation*}
for $u \in S^{n-1}$. Then $p_i$ is of degree $m-n+1$ and $p_i(u_j)=0$ for $j \neq i$. If $p_i(u_i)=0$, then $u_i \in A_i$ and we obtain a contradiction as before. Hence $p_i(u_i) > 0$. 
Due to \eqref{form2} and the assumption on coinciding moments, we obtain that
\begin{equation}\label{Eq}
\alpha_i p_i(u_i)=\sum_{j=1}^m \alpha_j p_i(u_j) = \sum_{j=1}^m \beta_j p_i(u_j)=\beta_i p_i(u_i).
\end{equation}
Since $p_i(u_i) > 0$, Equation~\eqref{Eq} implies that $\alpha_i = \beta_i$ for $i=1, \dots, n+1$. 

For $i=n+2, \dots, m$, define the polynomial
\begin{equation*}
p_i(u)=\frac{p(u)}{(1-\IP{u}{u_i})}
\end{equation*}
for $u \in S^{n-1}$. Then $p_i$ is of degree $m-n+1$ and $p_i(u_j)=0$ for $j \neq i$. If $p_i(u_i)=0$, then $u_i \in A_j$ for $j=1, \dots, n+1$, which is a contradiction. Hence, $p_i(u_i)=0$. By arguments as before, we obtain that $\alpha_i=\beta_i$ for $i=n+1, \dots, m$. Hence $\nu=\mu$, which yields \eqref{item02}.

The statement \eqref{item01} can be proved in a similar manner using the polynomials
\begin{equation*}
p(u)=\prod_{j=1}^m (1-\IP{u}{u_j})
\end{equation*}
and 
\begin{equation*}
p_i(u)=\frac{p(u)}{1-\IP{u}{u_i}}
\end{equation*}
for $u \in S^{n-1}$ and $i=1, \dots, m$.
\end{proof}

\begin{theorem} \label{Thm_uniqueness} 
Let $m \geq n+1$. A polytope $P \in \Poly[m]$ with nonempty interior is uniquely determined up to translation in $\K^n$ by its surface tensors up to rank $m-n+2$. If $n=2$, then the result holds for any $P \in \Poly[m]$.

The rank $m-n+2$ is optimal as there is a polytope $P_m \in \Poly[m]$ and a convex body $K_m \notin \Poly[m]$ having identical surface tensors up to rank $m-n+1$.
\end{theorem}

\begin{proof}
Let $P \in \Poly[m]$ have facet normals $u_1, \dots, u_m \in S^{n-1}$ and nonempty interior. Then, $ \supp S_{n-1}(P, \cdot) = \{u_1, \dots, u_m\}$ and $\aff \{u_1, \dots, u_m\} = \R^n$, so $S_{n-1}(P, \cdot)$ is uniquely determined in $\{S_{n-1}(K, \cdot) \mid K \in \K^n\} \subseteq \M$ by its moments up to order $m -n +2 $ due to Lemma~\ref{lemma_uniqueness} \eqref{item02}. Since the surface tensors of $P$ are rescaled versions of the moments of $S_{n-1}(P, \cdot)$, the first part of the statement follows as a convex body in $\R^n$ with nonempty interior is uniquely determined up to translation by its surface area measure. 
Now assume that $P \subseteq \R^2$ is a polytope in $\Poly[m]$ with empty interior. Then $P$ is contained in an affine hyperplane and $S_{n-1}(P, \cdot)=S(P)(\delta_u + \delta_{-u})$ for some $u \in S^{n-1}$. By Lemma~\ref{lemma_uniqueness} \eqref{item01}, the surface area measure of $P$ is uniquely determined by its moments up to second order. The second part of the statement then follows since any convex body in $\R^2$ is uniquely determined up to translation by its surface area measure.

To show that the rank $m-n+2$ cannot be reduced, we first consider the case $n=2$. For $m \geq 3$, let $P_m$ be a regular polytope in $\R^2$ with outer normals $u_j = (\cos(j \frac{2\pi}{m}), \sin(j\frac{2\pi}{m}))$ for $j=0, \dots m-1$ and facet lengths $\alpha_j=\frac{2\pi}{m}$ for $j=0, \dots, m-1$. Then, $P_m$ and the unit disc $B^2$ in $\R^2$ have identical surface tensors up to rank $m-1$. This is easily seen by calculating and comparing the harmonic intrinsic volumes of $P_m$ and $B^2$. 

Now, counter examples in $\R^n, n \geq 3$ can be constructed inductively. Essentially, if $P_{m-1}'$ and $K_{m-1}'$ are counter examples in $\R^{n-1}$, counter examples $P_m$ and $K_m$ in $\R^n$ are obtained as bounded cones with scaled versions of $P_{m-1}'$ and $K_{m-1}'$ as bases. More precisely, for a fixed $0<\alpha <1$ , define $f_{\alpha} : S^{n-2} \to S^{n-1}$ by $f_{\alpha}(u) = (\sqrt{1-\alpha^2} \,u, \alpha)$ for $u \in S^{n-2}$, and let
\begin{equation*}
\mu_m = f_{\alpha}(S_{n-1}(P_{m-1}', \cdot)) + \alpha S(P_{m-1}')\delta_{-e_3}
\end{equation*}
and
\begin{equation*}
\nu_m = f_{\alpha}(S_{n-1}(K_{m-1}', \cdot)) + \alpha S(K_{m-1}')\delta_{-e_3}.
\end{equation*}
By Minkowski's existence theorem, the measures $\mu_m$ and $\nu_m$ are surface area measures of convex bodies $P_m \in \Poly[m]$ and $K_m \in \K^n$, respectively. Direct calculations show that if $P_{m-1}'$ and $K_{m-1}'$ have identical surface tensors in $\R^{n-1}$ up to rank $(m-1)-(n-1)+1=m-n+1$, then $P_m$ and $K_m$ have identical surface tensors in $\R^n$ up to the same rank. Thus, we obtain that the rank $m-n+2$ is optimal in the sense that it cannot be reduced.
\end{proof}

Due to the one-to-one correspondence between surface tensors up to rank $s$ and harmonic intrinsic volumes up to degree $s$ of a convex body, the uniqueness result in Theorem~\ref{Thm_uniqueness} also holds if surface tensors are replaced by harmonic intrinsic volumes. 
 
\section{Stability results}\label{Sec_stab}

The shape of a convex body $K \in \K^n_n$ is uniquely determined by the set of surface tensors $\{\T[s] \mid s \in \N_0\}$ of $K$, but as described in the previous section, only the shape of polytopes are determined by a finite number of surface tensors. However, for an arbitrary convex body, a finite number of its surface tensors still contain information about its shape. This statement is quantified in this section, where we derive an upper bound of the translative Hausdorff distance between two convex bodies with a finite number of coinciding surface tensors.  

The cone of finite Borel measures $\M$ on $S^{n-1}$ is equipped with the Dudley metric
\begin{equation*}
\dudley{\mu}{\nu}=\sup \bigg\{\bigg \lvert \int_{S^{n-1}} f \, d(\mu - \nu)\bigg \rvert \biggm | \nBL[f] \leq 1 \bigg \} 
\end{equation*}
for $\mu, \nu \in \M$, where
\begin{equation*}
\nBL[f] = \nsup[f] + \nL[f] \quad \text{and} \quad  \nL[f]= \sup_{u \neq v} \frac{{\abs[f(u) - f(v)]}}{\norm[u-v]}
\end{equation*}
for any function $f \colon S^{n-1} \to \R$. It can be shown that the Dudley metric induces the weak topology on $\M$ (the case of probability measures is treated in \cite[Sec. 11.3]{Dudley2002} and is easily generalized to finite measures on $S^{n-1}$)  The set of real-valued functions on $S^{n-1}$ with $\nBL[f] < \infty$ is denoted $\BL$. Further, we let the vector space $L^2(S^{n-1})$ of square integrable functions on $S^{n-1}$ with respect to the spherical Lebesgue measure $\sigma$ be equipped with the usual inner product $\IPto{\cdot}{\cdot}$ and norm $\nto{\cdot}$.

As in \cite[Chap. 2.8.1]{Atkinson2012}, for $k \in \N$, we define the operator $\proj$ on the space $L^2(S^{n-1})$ by 
\begin{equation}\label{def_proj}
(\proj f)(u) = E_{nk} \int_{S^{n-1}} \bigg(\frac{1+ \IP{u}{v}}{2} \bigg)^k f(v) \sigma (dv)
\end{equation}
for $f \in L^2(S^{n-1})$ where the constant
\begin{equation*}
E_{nk}=\frac{(k+n-2)!}{(4\pi)^{\frac{n-1}{2}}\Gamma(k+\frac{n-1}{2})}
\end{equation*}
satisfies
\begin{equation}\label{E}
E_{nk}\int_{S^{n-1}} \bigg(\frac{1+ \IP{u}{v}}{2} \bigg)^k \sigma (du)=1.
\end{equation}
As $(1+\IP{u}{v})^k$ is a polynomial in $\IP{u}{v}$ of order $k$, it follows from the addition theorem for spherical harmonics (see, e.g., \cite[Thm. 3.3.3]{Groemer1996}) that the function $\proj f$ for $f \in L^2(S^{n-1})$ can be expressed as a linear combination of spherical harmonics of degree $k$ or less, see also \cite[pp. 61-62]{Atkinson2012}. More precisely, there are real constants $(a_{nkj})$ such that
\begin{equation}\label{lincomb}
\proj f = \sum_{j=0}^k a_{nkj} P_{nj} f,
\end{equation}
where $P_{nj} f$ is the projection of $f$ onto the space $\sH[j]$ of spherical harmonics of degree $j$. The constants in the linear combination \eqref{lincomb} are given by
\begin{equation*}
a_{nkj}=\frac{k!(k+n-2)!}{(k-j)!(k+n+j-2)!},
\end{equation*}
see \cite[p. 62]{Atkinson2012}. By \cite[Thm. 2.30]{Atkinson2012}, for any continuous function $f \colon S^{n-1} \to \R$, the sequence $(\proj f)_{k \in \N}$ converges uniformly to $f$ when $k \rightarrow \infty$. When $f \in \BL$, Lemma~\ref{Lemma_uniformbound} provides an upper bound for the convergence rate in terms of $\nL[f]$ and $\nsup[f]$.

\begin{lemma}\label{Lemma_uniformbound}
Let $0< \varepsilon < 1$ and $k \in \N$. For $f \in \BL$, we have
\begin{equation}\label{bound}
\nsup[\proj f - f] \leq \sqrt{k}^{\varepsilon-1}\nL[f] + 2 \omega_n E_{nk} \exp(-\frac{1}{4}k^{\varepsilon}) \nsup[f].
\end{equation}
\end{lemma}
\begin{proof}
We proceed as in the proof of \cite[Thm. 2.30]{Atkinson2012}. Let $f \in \BL$. Using \eqref{def_proj} and \eqref{E}, we obtain that
\begin{align*}
\abs[(\proj f)(u) - f(u)] &\leq E_{nk} \int_{S^{n-1}} \bigg(\frac{1+ \IP{u}{v}}{2} \bigg)^k \abs[f(u) - f(v)] \, \sigma(dv)
\\
&\leq
I_1(\delta,u) + I_2(\delta,u)
\end{align*}
for $ u \in S^{n-1}$ and $0 < \delta < 2 $, where
\begin{equation*}
I_1(\delta,u)=E_{nk} \int_{\{v \in S^{n-1} : \norm[u-v] \leq \delta \} } \bigg(\frac{1+ \IP{u}{v}}{2} \bigg)^k \abs[f(u) - f(v)] \, \sigma(dv)
\end{equation*}
and
\begin{equation*}
I_2(\delta,u)=E_{nk} \int_{\{ v \in S^{n-1} : \norm[u-v] > \delta \} } \bigg(\frac{1+ \IP{u}{v}}{2} \bigg)^k \abs[f(u) - f(v)] \, \sigma(dv).
\end{equation*}
Since $I_1(\delta,u) \leq \delta \nL[f]$ and 
\begin{equation*}
I_2(\delta, u) \leq 2 \omega_n E_{nk} \bigg(1-\frac{\delta^2}{4}\bigg)  \nsup[f],
\end{equation*}
we obtain that
\begin{equation}\label{bound1}
\abs[(\proj f)(u) - f(u)] \leq \delta \nL[f] + 2 \omega_n E_{nk} \bigg(1-\frac{\delta^2}{4}\bigg)^k \nsup[f].
\end{equation}
To derive the upper bound on $I_2$, we have used that $\IP{u}{v} = 1 - \frac{\norm[u-v]^2}{2}$ for $u,v \in S^{n-1}$. 

Now let $\delta=\sqrt{k}^{\varepsilon-1}$. From the mean value theorem, we obtain that
\begin{equation*}
\ln \bigg(1- \frac{\delta^2}{4} \bigg)^k = - \frac{1}{4} k^{\varepsilon}\frac{\ln (1) - \ln( 1 - \frac{1}{4}k^{\varepsilon-1})}{\frac{1}{4}k^{\varepsilon-1}} = - \frac{1}{4} k^\varepsilon \xi_k^{-1}
\end{equation*}
for some $\xi_k \in [1-\frac{1}{4}k^{\varepsilon-1},1]$. Hence,
\begin{equation}\label{bound2}
\bigg(1- \frac{\delta^2}{4} \bigg)^k  \leq \exp(-\frac{1}{4}k^{\varepsilon}).
\end{equation}
Combining \eqref{bound1} and \eqref{bound2} yields the assertion.
\end{proof}

\begin{remark}\label{Rem_convergence_rate}
Stirling's formula, $\Gamma(x) \sim \sqrt{2\pi}x^{x-\frac{1}{2}} \operatorname{e} ^{-x}$ for $x \rightarrow \infty$, implies that
\begin{equation*}
E_{nk} \sim \bigg( \frac{k}{4\pi} \bigg)^{\frac{n-1}{2}}
\end{equation*}
for $k \rightarrow \infty$. Hence, the upper bound in \eqref{bound} converges to zero for $k \rightarrow \infty$.
The choice of $\delta$ in the proof of Lemma~\ref{Lemma_uniformbound} is optimal in the sense that if we use $0  < \delta \leq \frac{c}{\sqrt{k}}$ with a constant $c > 0$, then the derived upper bound in $\eqref{bound}$ does not converge to zero. This follows as
\begin{equation*}
1 \geq \left( 1 - \frac{\delta^2}{4} \right)^k \geq  \left( 1 - \frac{c}{4k} \right)^k \rightarrow e^{-\frac{c}{4}}
\end{equation*}
for $k \rightarrow \infty$, when $0  < \delta \leq \frac{c}{\sqrt{k}}$.
\end{remark}

For functions $f \in \BL$ satisfying $\nBL[f] \leq 1$, Lemma~\ref{Lemma_uniformbound} yields an uniform upper bound, only depending on $k$ and the dimension $n$, of $\nsup[\proj f - f]$. In the following theorem, this is used to derive an upper bound of the Dudley distance between the surface area measures of two convex bodies where the harmonic intrinsic volumes up to a certain degree $s_o \in \N$ are close in $\R^{m_{s_o}}$.

\begin{theorem}\label{Thm_noise}
Let $K, L \in \K^n(R)$ for some $R > 0$ and let $s_o \in \N$. Let $0 < \varepsilon < 1$ and $\delta > 0$. If 
\begin{equation}\label{thm_noise_condition}
\sqrt{\omega_n m_{s_o}} \norm[\HV(K) - \HV(L)] \leq \delta
\end{equation}
then
\begin{equation}\label{Dudley_bound}
\dudley{S_{n-1}(K, \cdot)}{S_{n-1}(L, \cdot)} 
\leq 
c(n,R, \varepsilon)s_o^{\frac{\varepsilon-1}{2}} + \delta,
\end{equation}
where $c>0$ is a constant depending on $n, R$ and $\varepsilon$.
\end{theorem}

Due to the addition theorem for spherical harmonics, the condition \eqref{thm_noise_condition} is independent of the bases of $\sH$, $k \in \N$ that are used to derive the harmonic intrinsic volumes.

\begin{proof}[Proof of Theorem~\ref{Thm_noise}]
Let $f \in \BL$ with $\nBL[f] \leq 1$ and define the signed Borel measure $\nu=S_{n-1}(K, \cdot) - S_{n-1}(L, \cdot)$. Then, by \eqref{lincomb}, 
\begin{equation*}
\proj[s_o] f = \sum_{j=0}^{s_o} a_{ns_oj}\sum_{i=0}^{N(n,j)} \IPto{f}{H_{nji}} H_{nji},
\end{equation*}
where $\abs[a_{ns_oj}] \leq 1$. Since $\nto{f} \leq \sqrt{\omega_n} \nsup[f] \leq \sqrt{\omega_n}$, we obtain from Cauchy-Schwarz' inequality and a discrete version of Jensen's inequality that
\begin{align*}
\bigg\lvert \int_{S^{n-1}} \proj[s_o] f \, d\nu \bigg \rvert 
&\leq 
\sqrt{\omega_n} \sum_{j=0}^{{s_o}} \sum_{i=0}^{N(n,j)} \left \vert \int_{S^{n-1}} H_{nji} \, d\nu \right \vert 
\\
&\leq \bigg(\omega_n \bigg(\sum_{l=0}^{s_o} N(n,l) \bigg) \sum_{j=0}^{s_o} \sum_{i=0}^{N(n,j)} \bigg(\int_{S^{n-1}} H_{nji} \, d\nu\bigg)^2 \bigg)^{\frac{1}{2}}
\\
&=\sqrt{\omega_n m_{s_o}} \, \norm[\HV(K) - \HV(L)].
\end{align*}
Hence,
\begin{align*}
\bigg \rvert \int_{S^{n-1}} f \,d\nu \bigg \lvert &\leq  \bigg \lvert\int_{S^{n-1}} {\proj[s_o]} f - f \,d \nu \bigg \rvert + \bigg\lvert \int_{S^{n-1}} \proj[s_o] f \, d\nu \bigg \rvert
\\
&\leq
2R^{n-1}\omega_n(s_o^{\frac{\varepsilon-1}{2}} + 2 \omega_n E_{ns_o} \exp(-\frac{1}{4}s_o^{\varepsilon})) + \delta,
\end{align*}
where we used Lemma~\ref{Lemma_uniformbound} and that $\max\{S_{n-1}(K, S^{n-1}), S_{n-1}(L, S^{n-1})\} \leq R^{n-1}\omega_n$. For $k \rightarrow \infty$, the convergence of $E_{nk} \exp(-\frac{1}{4}k^{\varepsilon})$ to zero is faster than the convergence of $k^{\frac{\varepsilon-1}{2}}$, see Remark~\ref{Rem_convergence_rate}. This implies the existence of a constant $c$ only depending on $n, R$ and $\varepsilon$ satisfying \eqref{Dudley_bound}.
\end{proof}

\begin{corollary}\label{Cor_bound_dudley}
Let $K, L \in \K^n(R)$ for some $R > 0$ and let $0 < \varepsilon < 1$. If $\T=\Phi_{n-1}^s(L)$ for $0 \leq s \leq s_o$, then
\begin{equation*}
\dudley{S_{n-1}(K, \cdot)}{S_{n-1}(L, \cdot)} 
\leq
c(n,R, \varepsilon)s_o^{\frac{\varepsilon-1}{2}},
\end{equation*}
where $c > 0$ is a constant depending on $n,R$ and $\varepsilon$.
\end{corollary}  

\begin{proof}
The assumption that $K$ and $L$ have coinciding surface tensors up to rank $s_o$ implies that $\norm[\HV(K) - \HV(L)]=0$. The result then follows from Theorem~\ref{Thm_noise} with $\delta=0$.
\end{proof}

The translative Hausdorff distance between two convex bodies in $\K^n(r,R)$ admits an upper bound expressed by the $n$'th root of the Prokhorov distance between their surface area measures, see \cite[Thm. 8.5.3]{Schneider14}. Further, the Prokhorov distance between two Borel measures on $S^{n-1}$ can be bounded in terms of the square root of the Dudley distance between the measures. Therefore, Corollary~\ref{Cor_bound_dudley} in combination with \cite[Thm. 8.5.3]{Schneider14} and \cite[Lemma 9.5]{Gardner2006} yield the following stability result.

\begin{theorem}\label{thm_stability}
Let $K,L \in \K^n(r,R)$ for some $0 < r < R$ and let $0 < \varepsilon < 1$. If $\T = \Phi_{n-1}^s(L)$ for $0 \leq s \leq s_o$, then
\begin{equation*}
\delta^t(K,L) \leq c(n,r,R,\varepsilon)s_o^{-\frac{1 - \varepsilon}{4n}}
\end{equation*}
for a constant $c > 0
$ depending on $n,r,R$ and $\varepsilon$.
\end{theorem}

\section{Reconstruction of shape from surface tensors}\label{Sec_Recon}
In this section, we derive an algorithm that approximates the shape of an unknown convex body $K \in \K^n_n$ from a finite number of surface tensors $\{\T \mid 0 \leq s \leq s_o\}$ of $K$ for some $s_o \in \N$. The reconstruction algorithm is a generalization to higher dimension of Algorithm Surface Tensor in \cite{Kousholt2016} that reconstructs convex bodies in $\R^2$ from surface tensors. The shape of a convex body $K$ in $\R^n$ is uniquely determined by the surface tensors of $K$, when $K$ has nonempty interior, see \cite[Sec. 4, p. 10]{Kousholt2016}. For $n=2$, the surface tensors of $K$ determine the shape of $K$ even when $K$ is lower dimensional. Therefore, the algorithm in \cite{Kousholt2016} can be used to approximate the shape of arbitrary convex bodies in $\R^2$, whereas the algorithm described in this section only allows for convex bodies in $\R^n$ with nonempty interior. A non-trivial difference between the algorithm in the two-dimensional setting and the generalized algorithm is that in higher dimension, it is crucial that the first and second order moments of a Borel measure $\mu$ on $S^{n-1}$ determine if $\mu$ is the surface area measure of a convex body. Therefore, this is shown in Lemma~\ref{lemma_supportmeasures} that is based on the following remark. 


\begin{remark}\label{rem_support}
Let $\mu$ be a Borel measure on the unit sphere $S^{n-1}$. Then,
\begin{equation}\label{integral}
\int_{S^{n-1}} \IP{z}{u}^2 \mu (du) > 0
\end{equation}
for all $z \in S^{n-1}$ if and only if the support of $\mu$ is full-dimensional (meaning that the support of $\mu$ is not contained in any great subsphere of $S^{n-1}$). As the integral in \eqref{integral} 
is determined by the second order moments
\begin{equation*}
m_{ij}(\mu) = \int_{S^{n-1}} u_i u_j \, \mu (du)
\end{equation*}
of $\mu$, these moments determine if the support of $\mu$ is full-dimensional. More precisely, the support of $\mu$ is full-dimensional if and only if the matrix of second order moments $M(\mu)=\{m_{ij}(\mu)\}_{i,j=1}^n$ is positive definite as   
\begin{equation*}
z^{\top} M(\mu) z = \int_{S^{n-1}} \IP{z}{u}^2 \mu (du)
\end{equation*}
for $z \in \R^{n}$.
\end{remark}

\begin{lemma}\label{lemma_supportmeasures}
Let $\mu$ be a Borel measure on $S^{n-1}$ with $\mu(S^{n-1})> 0$.
\begin{enumerate}[(i)]
\item\label{item1support} The measure $\mu$ is the surface area measure of a convex body $K \in \K^n_n$, if and only if the first order moments of $\mu$ vanish and the matrix $M(\mu)$ of second order moments of $\mu$ is positive definite.

\item\label{item2support} The measure $\mu$ is the surface area measure of a convex body $K \in \K^n \setminus \K^n_n$ if and only if the first order moments of $\mu$ vanish and the matrix $M(\mu)$ of second order moments of $\mu$ has one positive eigenvalue and $n-1$ zero eigenvalues. 
\end{enumerate}

In the case, where \eqref{item2support} is satisfied, the measure $\mu$ is the surface area measure of every convex body $K$ with surface area $\frac{1}{2}\mu(S^{n-1})$ contained in a hyperplane with normal vector $u$, where $u \in S^{n-1}$ is a unit eigenvector of $M(\mu)$ corresponding to the positive eigenvalue (u is unique up to sign).

\end{lemma}

\begin{proof}
Remark~\ref{rem_support} implies that the interior of a convex body $K$ is nonempty if and only if the matrix of second order moments of $S_{n-1}(K, \cdot)$ is positive definite, so the statement \eqref{item1support} follows from Minkowski's existence theorem, \cite[Thm. 8.2.2]{Schneider14}. 

If $\mu$ is the surface area measure of $K \in \K^n \setminus \K^n_n$, then $\mu$ is of the form 
\begin{equation*}
\mu= \frac{\mu(S^{n-1})}{2}(\delta_u + \delta_{-u})
\end{equation*}
for some $u \in S^{n-1}$. Then, the first order moments of $\mu$ vanish, and the matrix $M(\mu)$ of second order moments of $\mu$ is  $\mu(S^{n-1}) u^2$. Hence, $M(\mu)$ has one positive eigenvalue $\mu(S^{n-1})$ with eigenvector $u$ and $n-1$ zero eigenvalues.

If the matrix $M(\mu)$ is positive semidefinite with one positive eigenvalue $\alpha > 0$ and $n-1$ zero eigenvalues, then $M(\mu)=\alpha u^2$, where $u \in S^{n-1}$ is a unit eigenvector (unique up to sign) corresponding to the positive eigenvalue. Assume further that the first order moments of $\mu$ vanish, and define the measure $\nu=\frac{\alpha}{2}(\delta_u + \delta_{-u})$. Then $\mu$ and $\nu$ have identical moments up to order $2$, and Lemma~\ref{lemma_uniqueness} \eqref{item01} yields that $\mu=\nu$. Therefore, $\mu$ is the surface area measure of any convex body $K$ with surface area $\alpha$ contained in a hyperplane with normal vector $u$.
\end{proof}

\subsection{Reconstruction algorithm based on surface tensors}\label{subsec_recon}

Let $K_0 \in \K^n_n$ be fixed. We consider $K_0$ as unknown and assume that the surface tensors $\Tnul[0], \dots, \Tnul[s_o]$ of $K_0$ are known up to rank $s_o$ for some natural number $s_o \geq 2$. The aim is to construct a convex body with surface tensors identical to the known surface tensors of $K_0$. We proceed as in \cite[Sec. 5.1]{Kousholt2016}. 

Let
\begin{equation}\label{def_Ms}
  \Ms=\{(\alpha, \textbf{u}) \in \R^{m_{s_o}} \times (S^{n-1})^{m_{s_o}} \mid \alpha_j \geq 0, \quad \sum_{j=1}^{m_{s_o}} \alpha_j u_j = 0 \},
 \end{equation}   
and consider the minimization problem
\begin{equation}\label{minimization}
\min_{(\alpha, \textbf{u}) \in \Ms} \sum_{j=1}^{m_{s_o}} \bigg( \phi_{n-1}^{s_o}(K_0)_j - \sum_{i=1}^{m_{s_o}}  \alpha_i g_{s_oj}(u_i)\bigg)^2,
\end{equation}
where $g_{s_oj} \colon S^{n-1} \to \R$ is the polynomial that satisfies that
\begin{equation*}
\int_{S^{n-1}} g_{s_oj}(u) \, S_{n-1}(K_0, du) = \phi_{n-1}^{s_o}(K_0)_j
\end{equation*}
for $j=1, \dots, m_{s_o}$.
Notice, that the objective function in \eqref{minimization} is known, as the surface tensors $\Tnul[s_o-1]$ and $\Tnul[s_o]$ are assumed to be known. By \cite[Thm. 4.1]{Kousholt2016}, there exists a polytope P (not necessarily unique) with at most $m_{s_o}$ facets and surface tensors identical to the surface tensors of $K_0$ up to rank $s_o$. Now, let $v_1, \dots, v_{m_{s_o}} \in S^{n-1}$ be the outer normals of the facets of such a polytope $P$ and $a_1, \dots, a_{m_{s_o}} \geq 0$ be the corresponding $(n-1)$-dimensional volumes of the facets.  If $P$ has $k < m_{s_o}$ facets, then $a_{k+1}= \dots =a_{m_{s_o}}=0$.
Then $S_{n-1}(P, \cdot)= \sum_{j=1}^{m_{s_o}} a_j \delta_{v_j}$, and
\begin{equation*}
\phi_{n-1}^{s_o}(P)_j = \sum_{i=1}^{m_{s_o}} a_i g_{s_oj}(v_i).
\end{equation*}
As $P$ and $K_0$ has identical surface tensors up to rank $s_o$, this implies that
\begin{equation}\label{minimum_is_0}
\sum_{j=1}^{m_{s_o}} \bigg( \phi_{n-1}^{s_o}(K_0)_j - \sum_{i=1}^{m_{s_o}}  a_i g_{s_oj}(v_i)\bigg)^2 = 0.
\end{equation}
Therefore,
$
(a, \textbf{v})=(a_1, \dots, a_{m_{s_o}}, v_1, \dots, v_{m_{s_o}}) \in \Ms
$
is a solution to the minimization problem \eqref{minimization}.

Now, let $(\alpha,\textbf{u}) \in \Ms$ be an arbitrary solution to \eqref{minimization} and define the Borel measure $\varphi=\sum_{i=1}^{m_{s_o}} \alpha_i \delta_{u_i}$ on $S^{n-1}$. As the minimum value of the objective function is $0$ due to \eqref{minimum_is_0}, the moments of $\varphi$ and $S_{n-1}(K_0, \cdot)$ of order $s_o-1$ and $s_o$ are identical. This implies that the moments of $\varphi$ and $S_{n-1}(K_0,\cdot)$ of order $1$ and $2$ are identical as $s_o \geq 2$, see \cite[Remark 3.1]{Kousholt2016}. Then Lemma~\ref{lemma_supportmeasures} \eqref{item1support} yields the existence of a polytope $Q \in \Poly$ with nonempty interior such that $S_{n-1}(Q, \cdot)= \varphi$. The surface tensors of $Q$ are identical to the surface tensors of $K_0$ up to rank $s_o$. 

In the two-dimensional setup in \cite[Sec. 5.1]{Kousholt2016}, every vector in $\Ms$ corresponds to the surface area measure of a polytope. In the $n$-dimensional setting, this is not the case, as Minkowski's existence theorem requires that the linear hull of the vectors $\alpha_1 u_1, \dots, \alpha_{m_{s_o}}u_{m_{s_o}}$ is $\R^n$, when $n>2$. However, as the above considerations show, Lemma~\ref{lemma_supportmeasures} ensures that every solution vector to the minimization problem~\eqref{minimization}, in fact, corresponds to the surface area measure of a polytope, which is sufficient to obtain a polytope with the required surface tensors.

The minimization problem~\eqref{minimization} can be solved numerically, and a polytope corresponding to the solution $(\alpha, \textbf{u}) \in \Ms$ can be constructed using Algorithm MinkData described in \cite{Lemordant1993}, (see also \cite[Sec. A.4]{Gardner06}). This polytope has surface tensors identical to the surface tensors of $K_0$ up to rank $s_o$.
	
\paragraph*{Algorithm Surface Tensor ($n$-dim)}
\begin{description}
\item[Input:] A natural number $s_o \geq 2$ and surface tensors $\Phi_{n-1}^{s_o-1}(K_0)$ and $\Phi_{n-1}^{s_o}(K_0)$ of an unknown convex body $K_0 \in \K^n_n$.

\item[Task:] Construct a polytope $\mK$ in $\R^n$ with at most $m_{s_o}$ facets such that $\mK$ and $K_0$ have identical surface tensors up to rank $s_o$.

\item[Action:]
Find a vector $(\alpha,\textbf{u}) \in \Ms$ that minimizes
\begin{equation*}
\sum_{j=1}^{m_{s_o}} \bigg( \phi_{n-1}^{s_o}(K_0)_j - \sum_{i=1}^{m_{s_o}}  \alpha_i g_{s_oj}(u_i)\bigg)^2.
\end{equation*}
The vector $(\alpha,\textbf{u})$ describes a polytope $\mK$ in $\R^n$ with at most $m_{s_o}$ facets. Reconstruct $\mK$ from $(\alpha,\textbf{u})$ using Algorithm MinkData.

\end{description}

\begin{remark}\label{rem_robust}
Solving the minimization problem \eqref{minimization} numerically might introduce small errors, such that the surface tensors $\Phi_{n-1}^{s_o-1}(\mK)$ and $\Phi_{n-1}^{s_o}(\mK)$ are only approximations of the surface tensors $\Tnul[{s_o-1}]$ and $\Tnul[s_o]$. Small errors in the surface tensors of rank $s_o-1$ and $s_o$ imply the risk of huge errors in the surface tensors of rank less than $s_o$. This follows from the way the surface tensors $\Phi_{n-1}^s$, $0 \leq s \leq s_o$ are related to the surface tensors $\Phi_{n-1}^{s_o-1}$ and $\Phi_{n-1}^{s_o}$ as described in Section~\ref{Sec_prelim}, see \eqref{Constant_c}. The main problem is the constant 
\begin{equation*}
c_{s,s_o}=\frac{s_o! \, \omega_{s_o}}{s! \, \omega_{s+1}}
\end{equation*}
that increases rapidly with $s_o$ for fixed $s$, and therefore might cause huge errors in, for instance, the surface area of $\mK$. The algorithm can be made more robust to numerical errors by replacing the surface tensors with the scaled versions $(s! \omega_{s+1})^{-1}\Phi_{n-1}^s$ of the surface tensors. The two versions of the algorithm are theoretically equivalent.
\end{remark}

\subsection{Consistency of the reconstruction algorithm}
The output of the algorithm described in the previous section is a polytope with surface tensors identical to the surface tensors of $K_0$ up to a given rank $s_o$. In this section, we show that for large $s_o$ the shape of the output polytope is a good approximation of the shape of $K_0$.

For each $s_o \geq 2$, let $\mK$ be an output of the algorithm based on surface tensors up to rank $s_o$. Then exist $r_{s_o}, R_{s_o} > 0$ such that $\mK, K_0 \in \K^n(r_{s_o}, R_{s_o})$ and by Theorem~\ref{thm_stability}, we obtain
\begin{equation*}
\delta^t(K_0, \mK) \leq c(n,r_{s_o}, R_{s_o}, \epsilon)s_o^{-\frac{1-\varepsilon}{4n}}
\end{equation*}
for $\varepsilon > 0$. Notice that $c$ depends on $s_o$ through $r_{s_o}$ and $R_{s_o}$, so even though the factor $s_o^{-1/(4n) + \varepsilon}$ converges to $0$ when $s_o$ increases, we do not immediately obtain the wanted consistency. To prevent the dependence of $c$ on $s_o$, we show that there exist radii $r,R > 0$ such that $K_0, \mK \in \K^n(r, R)$ for each $s_o \geq 2$.
  
Lemma~\ref{lemma_supportmeasures} yields that the surface tensor $\T[2]$ of a convex body $K \in \K^n$ determines if $K$ has nonempty interior. In Lemma~\ref{spheres}, we show that $\T[2]$ even determines the radius of a sphere contained in $K$ and the radius of a sphere containing $K$, when $K$ has nonempty interior. 

For a convex body $K \in \K^n_n$, the coefficient matrix $\{\T[2](e_i,e_j)\}_{i,j=1}^n$ of $\T[2]$ is symmetric and positive definite, and has therefore $n$ positive eigenvalues. In the following, we let $\lambda_{min}(K) > 0$ denote the smallest of these eigenvalues. The proof of Lemma~\ref{spheres} is inspired by the proof of \cite[Lemma 4.4.6]{Gardner06}.

\begin{lemma}\label{spheres}
Let $K \in \K^n_n$ with centre of mass at the origin. Let
\begin{equation}\label{rR}
R=\frac{S(K)}{4\pi \lambda_{min}(K)} \bigg(\frac{S(K)}{\omega_n}\bigg)^{\frac{1}{n-1}} \quad \text{and} \qquad  r=\frac{2\pi \lambda_{min}(K)}{(n+1)(4R)^{n-2}}.
\end{equation}
Then $rB^n \subseteq K \subseteq RB^n$. 
\end{lemma}

\begin{proof}
Let $x$ be a point on the boundary $\partial K$ of $K$. Then $\norm[x] > 0$, so $v = \tfrac{x}{\norm[x]} \in S^{n-1}$ is well-defined. By monotonicity and positive multilinearity of mixed volumes (see, e.g., \cite[(A.16),(A.18)]{Gardner06}) and the isoperimetric inequality (see, e.g., \cite[(B.14)]{Gardner06}), we obtain that
\begin{equation}\label{eqxMixedV}
\norm[x]V(K, n-1; [o,v]) = V(K, n-1; [o,x]) 
\leq V_n(K) \leq \bigg(\frac{S(K)}{\omega_n}\bigg)^{\frac{n}{n-1}}\kappa_n, 
\end{equation}
where $V$ is the mixed volume, $V_n$ is the $n$-dimensional volume and $[a,b]$ is the convex hull of $\{a,b\} \subseteq \R^n$.
Further, we have that
\begin{align*}
V(K, n-1; [o,v])&= \frac{1}{n}\int_{S^{n-1}} h_{[0,v]}(u) \, S_{n-1}(K, du) 
\\&= \frac{1}{2n}\int_{S^{n-1}} \abs[\IP{u}{v}] \, S_{n-1}(K, du) 
\\
&\geq \frac{1}{2n}\int_{S^{n-1}} \IP{u}{v}^2 \, S_{n-1}(K, du)
= \frac{4\pi}{n} \T[2](v,v),
\end{align*}
where we have used \cite[(A.11) and (A.12)]{Gardner06} and that $S_{n-1}(K, \cdot)$ has centroid at the origin.
Hence, 
\begin{equation}\label{eqMixedV}
V(K, n-1; [o,v]) \geq \frac{4 \pi}{n} \lambda_{min}(K).
\end{equation}
Equations \eqref{eqxMixedV} and \eqref{eqMixedV} yield that $\norm[x] \leq R$, so $K \subseteq RB^n$. 

As the centre of mass of $K$ is at the origin, then \cite[p. 320, note 6]{Schneider14} and the references given there yield that
\begin{equation*}
\frac{1}{n+1} w(K,u) \leq h_K(u)
\end{equation*}
for $u \in S^{n-1}$, where $w(K, \cdot)$ is the width function of $K$. Since 
\begin{equation*}
w(K,u)=h_K(u)+h_K(-u)=h_{K_s}(u)
\end{equation*}
where $K_s=K + (-K)$, it is sufficient to show that $r(n+1)B^n \subseteq K_s$ in order to obtain that $rB^n \subseteq K$. Due to origin-symmetry of $K_s$, we can proceed as in the proof of \cite[Lemma 4.4.6]{Gardner06}. Let $c = \sup\{a > 0 \mid aB^n \subseteq K_s\} > 0$. Then $cB^n \subseteq K_s$ and $\partial K_s \cap \partial cB^n \neq \emptyset$. As $K_s$ and $cB^n$ are origin-symmetric there are contact points $z, -z \in \partial K_s \cap \partial cB^n$ and common parallel supporting hyperplanes of $K_s$ and $cB^n$ in $z$ and $-z$. By the first part of this proof, we have $K_s \subseteq 2RB^n$, so $K_s$ is contained in a $n$-dimensional box with one edge of length $2c$ parallel to $z$ and $n-1$ edges of length $4R$ orthogonal to $z$. More precisely, 
\begin{equation*}
K_s \subseteq \{x \in \R^n \mid \abs[\IP{x}{z}] \leq c \} \cap \bigcap_{j=2}^n \{x \in \R^n \mid \abs[\IP{x}{u_j}] \leq 2R\}
\end{equation*}
where $u_2, \dots, u_n \in S^{n-1}$ and $z$ form an orthogonal basis of $\R^n$. This implies that 
\begin{equation}\label{eqProj}
V_{n-1}(K_s \mid (u_2)^{\perp}) \leq 2c (4R)^{n-2},
\end{equation} 
where $K_s \mid (u_2)^{\perp}$ is the orthogonal projection of $K_s$ onto $(u_2)^{\perp}$. Using \cite[(A.37)]{Gardner06} and that Equation~\eqref{eqMixedV} holds for any $v \in S^{n-1}$, we obtain
\begin{align*}
V_{n-1}(K_s \mid (u_2)^{\perp}) &\geq V_{n-1}(K \mid (u_2)^{\perp}) 
\\
&= nV(K,n-1; [o,u_2]) \geq 4 \pi \lambda_{min}(K),
\end{align*}
so from \eqref{eqProj} it follows that
\begin{equation*}
c \geq \frac{2\pi \lambda_{min}(K)}{(4R)^{n-2}},
\end{equation*}
which yields that $r(n+1)B^n \subseteq K_s$.
\end{proof}

\begin{theorem}\label{thm_conv}
Let $K_0 \in \K^n_n$, $s_o \geq 2$ be a natural number and $0 < \varepsilon < 1$. If the surface tensors up to rank $s_o$ of a convex body $K_{s_o}$ coincide with the surface tensors of $K_0$, then
\begin{equation}\label{cons_bound}
\delta^t(K_0,K_{s_o}) \leq c(n,\varepsilon, \Tnul[2])s_o^{-\frac{1- \varepsilon}{4n}},
\end{equation}
where $c>0$ is a constant depending only on $n, \varepsilon$ and $\Tnul[2]$. Hence, if $(K_{s_o})_{s_o \in \N}$ is a sequence of convex bodies satisfying $\Tnul=\Phi_{n-1}^s(K_{s_o})$ for $0 \leq s \leq s_o$, then the shape of $K_{s_o}$ converges to the shape of $K_0$ when $s_o \to \infty$.
\end{theorem}

\begin{proof} 
When defined as in \eqref{rR} with $K$ replaced by $K_0$, the radii $r$ and $R$ are determined by $\Tnul[2]$, and since $\Tnul[2]=\Phi_{n-1}^{2}(K_{s_o})$, Lemma~\ref{lemma_supportmeasures} and Lemma~\ref{spheres} yield that $K_{s_o} + x_{s_o}, K_0 + x_0 \in \K^n(r,R)$ for suitable $x_{s_o},x_0 \in \R^n$. Then, using translation invariance of $\delta^t$, we obtain the bound~\eqref{cons_bound} from Theorem~\ref{thm_stability}. Now, the constant $c$ does not depend on $s_o$, so the stated convergence result is obtained from $\eqref{cons_bound}$.
\end{proof}

The consistency of Algorithm Surface Tensor ($n$-dim) follows from Theorem~\ref{thm_conv}.

\subsection{Examples: Reconstruction of convex bodies in $\R^3$}
In this section, we give two examples where Algorithm Surface Tensor is used to reconstruct the shape of a convex body in $\R^3$. Following Remark~\ref{rem_robust}, the scaled surface tensors $s! \omega_{s+1} \Phi_2^s$ have been used in order to make the reconstructions more robust to numerical errors. In the first example, the ellipsoid in Figure~\ref{Ellipsoid} is reconstructed. The reconstructions of the ellipsoid are based on surface tensors up to rank $s_o=2,4,6$, see Figure~\ref{EllipsoidRecon}. In the second example, the pyramid displayed in Figure~\ref{Pyramid} is reconstructed. The reconstructions of the pyramid are executed with $s_o=2,3,4$, see Figure~\ref{PyramidRecon}. 

\begin{figure}
\centering
\begin{minipage}[c]{0.47 \textwidth}
\centering
\includegraphics[width=4cm]{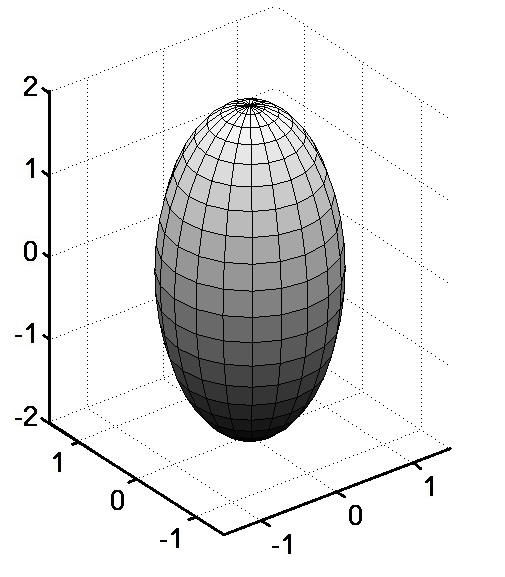}
\caption{Ellipsoid}
\label{Ellipsoid}
\end{minipage}
\hfill
\begin{minipage}[c]{0.47 \textwidth}
\centering
\includegraphics[width=5cm]{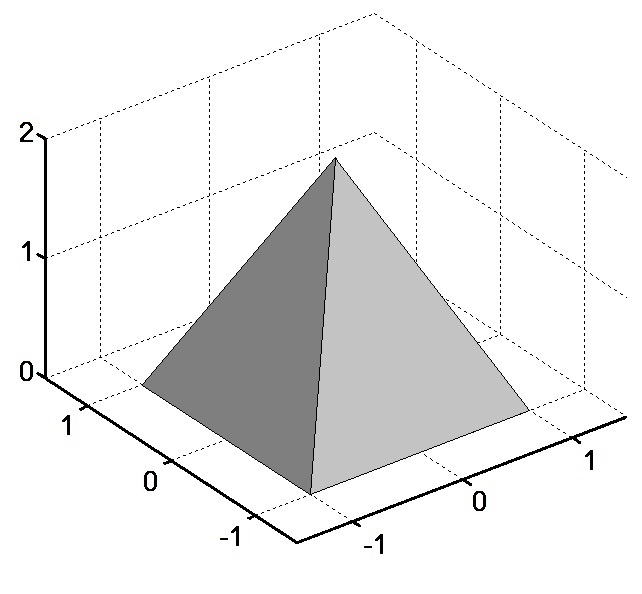}
\caption{Pyramid}
\label{Pyramid}
\end{minipage}
\end{figure}

The minimization problem \eqref{minimization} is solved by means of the \textit{fmincon} procedure provided by MatLab, and a polytope corresponding to the solution to \eqref{minimization} is reconstructed using Algorithm MinkData. This algorithm has been implemented by Gardner and Milanfar for $n \leq 3$, see \cite[Sec. A4]{Gardner06}, and for $n=3$ the algorithm has recently become available on the website www.geometrictomography.com run by Richard Gardner.

The surface tensor of rank $2$ of a convex body contains information of the main directions and the degree of anisotropy of the convex body. The effect of this is, in particular, visible in the plots in Figure~\ref{EllipsoidRecon} that show that the three reconstructions of the ellipsoid are elongated in the direction of the third axis.
As expected, the reconstructions of the ellipsoid and the reconstructions of the pyramid become more accurate when $s_o$ increases. The pyramid has $5$ facets, so according to Theorem~\ref{Thm_uniqueness}, the surface tensors up to rank $4$ uniquely determine the shape of the pyramid. The last plot in Figure~\ref{PyramidRecon} shows that the reconstruction based on surface tensors up to rank $4$ is indeed very precise. Deviation from the pyramid can be ascribed to numerical errors. 

\begin{figure}
\centering
\includegraphics[width=3.3cm]{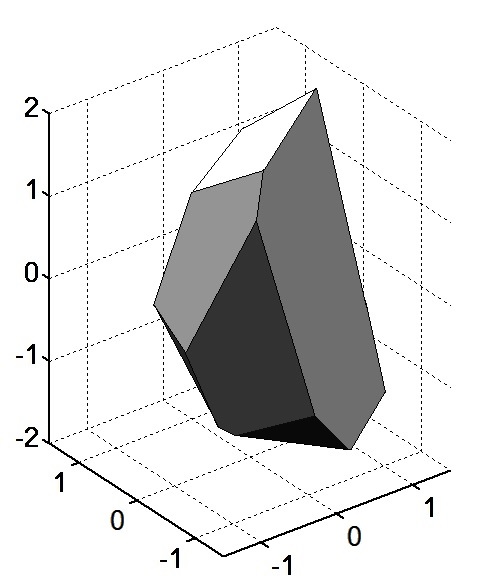}
\hspace*{0.7cm}
\includegraphics[width=3.3cm]{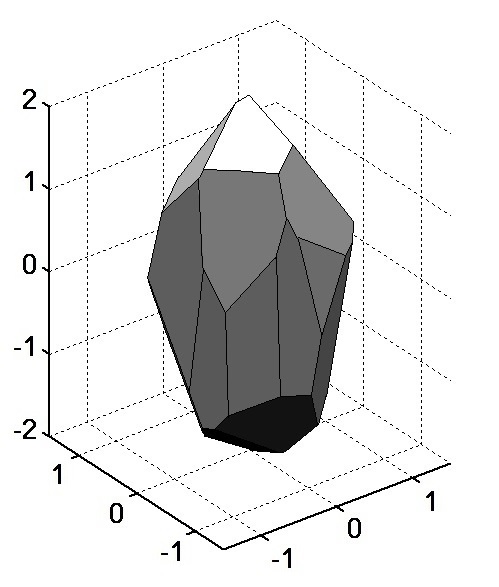}
\hspace*{0.7cm}
\includegraphics[width=3.3cm]{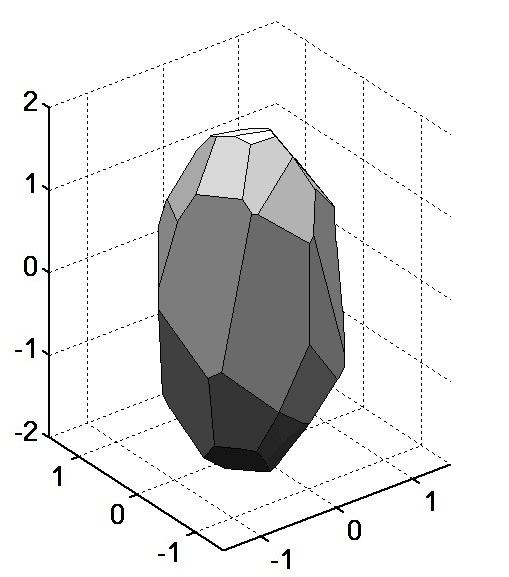}
\caption{Reconstructions of the ellipsoid in Figure~\ref{Ellipsoid} based on surface tensors up to rank $s_o=2,4,6$.}
\label{EllipsoidRecon}
\end{figure}

\begin{figure}
\centering
\includegraphics[width=3.8cm]{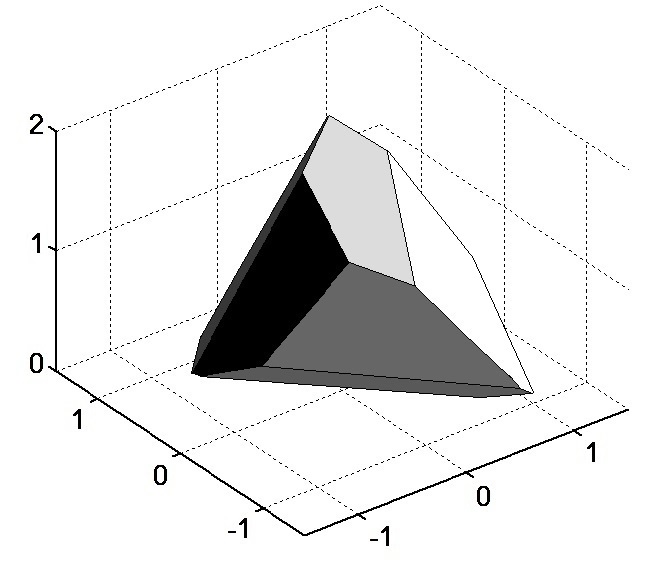}
\hspace*{0.7cm}
\includegraphics[width=3.8cm]{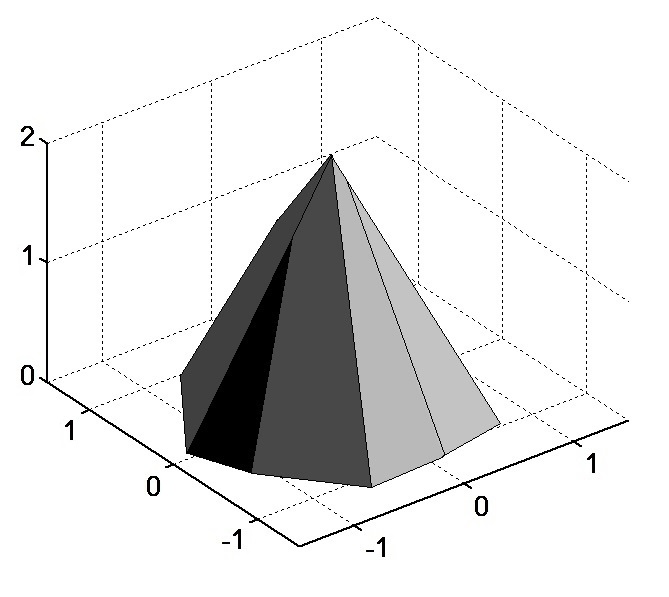}
\hspace*{0.7cm}
\includegraphics[width=3.8cm]{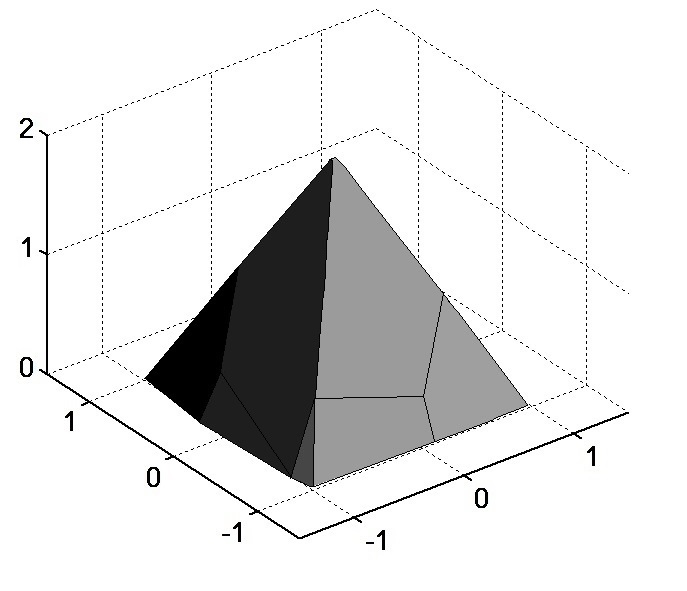}
\caption{Reconstructions of the pyramid in Figure~\ref{Pyramid} based on surface tensors up to rank $s_o=2,3,4$.}
\label{PyramidRecon}
\end{figure}

\section{Reconstruction of shape from harmonic intrinsic volumes}\label{Sec_ReconHVol}
Due to the correspondence between surface tensors and harmonic intrinsic volumes, a convex body $K \in \K_n^n$ is uniquely determined by the set of harmonic intrinsic volumes $\{\psi_{(n-1)sj}(K) \mid s \in \N_0, \, j=1, \dots, N(n,s) \}$ of $K$. In this section, we derive an algorithm that approximates the shape of an unknown convex body $K_0 \in \K^n_n$ from measurements subject to noise of a finite number of harmonic intrinsic volumes of $K_0$. The reconstruction algorithm we derive is a generalization to an $n$-dimensional setting of Algorithm Harmonic Intrinsic Volume LSQ described in \cite{Kousholt2016}.

\subsection{Reconstruction algorithm based on measurements of harmonic intrinsic volumes}\label{subsec_reconHIV}

Let $K_0 \in \K^n_n$ be an unknown convex body where measurements of the harmonic intrinsic volumes of $K_0$ are available up to degree $s_o \geq 2$. Due to noise, the measurements are of the form
$\HV(K_0) + \epsilon_{s_o}$, where $\epsilon_{s_o}$ is an $m_{s_o}$-dimensional vector of random variables with zero mean and finite variance. As the harmonic intrinsic volumes of degree $1$ of $K_0$ are known to vanish, these should not be measured, so we let the corresponding noise variables be $0$.

In Section~\ref{Sec_Recon}, the exact surface tensors of $K_0$ was known. In that situation, we constructed a convex body with the same surface tensors as $K_0$. In this section, only \emph{noisy measurements} of the harmonic intrinsic volumes are available, and it is typically no longer possible to construct a convex body with the exact same harmonic intrinsic volumes as $K_0$. Instead, the aim is to construct a convex body $\hat{K}_{s_o}^H \in \K^n$ such that the harmonic intrinsic volumes of $\hat{K}_{s_o}^H$ fit the measurements $\HV(K_0) + \epsilon_{s_o}$ of the harmonic intrinsic volumes of $K_0$ in a least squares sense. Hence, $\hat{K}_{s_o}^H$ should minimize the mapping $D_{s_o} \colon \K^n \to [0, \infty)$ defined as  
\begin{equation*}
D_{s_o}(K) =  \norm[\HV(K_0) + \epsilon_{s_o} - \HV(K)]^2
\end{equation*}
for $K \in \K^n$. In the $2$-dimensional setup, \cite[Lemma 6.1]{Kousholt2016} yields the existence of a convex body that minimizes $D_{s_o}$. In the $n$-dimensional setting, however, the existence of such a convex body can not be ensured. This existence problem is overcome by extending the domain of $D_{s_o}$ such that the mapping attains its infimum. This extension  prevents the existence problem and thus establishes a natural framework for reconstruction in the $n$-dimensional setting.

First notice that $D_{s_o}(K)$ only depends on $K \in \K^n$ through $S_{n-1}(K, \cdot)$, so a version $\check{D}_{s_o}$ of $D_{s_o}$ can be defined on the set $\{S_{n-1}(K, \cdot) \mid K \in \K^n\}$ letting $\check{D}_{s_o}(S_{n-1}(K, \cdot)) = D_{s_o}(K)$ for $K \in \K^n$. In the weak topology, the closure of $\{S_{n-1}(K, \cdot) \mid K \in \K^n\} \subseteq \M$ is the set 
\begin{equation*}
\M_0 =\biggl\{\mu \in \M \biggm \vert \int_{S^{n-1}} u \,  \mu (du) = 0 \biggr\},
\end{equation*}
and the domain of $\check{D}_{s_o}$ is extended to $\M_0$ by defining
\begin{equation*}
\check{D}_{s_o}(\mu) =\norm[\HV(K_0) + \epsilon_{s_o} - \HV(\mu)]^2
\end{equation*}
for $\mu \in \M_0$. 
Then
\begin{equation}\label{infinf}
\inf_{K \in \K^n} D_{s_o}(K) = \inf_{\mu \in \M_0} \check{D}_{s_o}(\mu)
\end{equation}
since $\check{D}_{s_o}$ is continuous on $\M_0$.

The infimum of $\check{D}_{s_o}$ is attained on $\M_0$, and in addition, it can be shown that $\check{D}_{s_o}$ is minimized by a measure in $\M_{m_{s_o}}$, where
\begin{equation*}
\M_k = \biggl\{ \mu \in \M_0 \mid \mu = \sum_{j=1}^k \alpha_j \delta_{u_j}, \, \alpha_j \geq 0, \,  u_j \in S^{n-1} \biggr \}
\end{equation*}
for $k \in \N$. This is the content of the following Lemmas~\ref{lemma_measure_existence} and \ref{lemma_infD}. Due to the close connection between $D_{s_o}$ and $\check{D}_{s_o}$, we write $D_{s_o}$ for both versions of the mapping. 

\begin{lemma}\label{lemma_measure_existence}
Let $\mu \in \M_0$ and $s \in \N_0$. Then there exist a measure $\mu_s \in \M_{m_s}$ such that $\mu$ and $\mu_s$ have identical moments up to order $s$.
\end{lemma}

The proof of Lemma~\ref{lemma_measure_existence} follows the lines of the proof of \cite[Thm. 4.1]{Kousholt2016}. The result also holds if $\M_0$ and $\M_{m_s}$ are replaced by the larger sets $\M$ and $\{ \mu \in \M \mid \mu = \sum_{j=1}^k \alpha_j \delta_{u_j}, \, \alpha_j \geq 0, \,  u_j \in S^{n-1}\}$.

\begin{lemma}\label{lemma_infD}
There exists a measure $\mu_{s_o} \in \M_{m_{s_o}}$ such that
\begin{equation}\label{Inf_D}
D_{s_o}(\mu_{s_o}) = \inf_{\mu \in \M_0} D_{s_o}(\mu).
\end{equation}
If $\mu_1, \mu_2 \in \M_0$ minimize $D_{s_o}$, then $\mu_1$ and $\mu_2$ have identical moments up to order $s_o$.
\end{lemma}
\begin{proof}
Let $H=\{\HV(\mu) \mid \mu \in \M_0\} \subseteq \R^{m_{s_o}}$. Then
\begin{equation*}
\inf_{\mu \in \M_0} D_{s_o}(\mu) = \inf_{x \in H} \norm[\HV(K_0) + \epsilon_{s_o} - x]^2.
\end{equation*}
Let $\{\HV[s](\mu_k)\}_{k \in \N}$ be a  convergent sequence in $H$. Then, $\sup_{k \in \N} \mu_k(S^{n-1}) < \infty$, since $\mu(S^{n-1})=\sqrt{\omega_n}\psi_{n01}(\mu)$ for $\mu \in \M_0$. Since $\M_0$ is closed, this implies that there exists a subsequence $(\mu_{k_l})_{l \in \N}$ of $(\mu_k)_{k \in \N}$ that converges weakly to a measure $\mu \in \M_0$, see \cite[Cor. 31.1]{Bauer2001}. Then $\HV(\mu_k) \to \HV(\mu)$ for $k \to \infty$ as spherical harmonics are continuous on $S^{n-1}$. Hence, $H$ is closed in $\R^{m_{s_o}}$. 
Solving the minimization problem
\begin{equation*}
\inf_{x \in H} \norm[\HV(K_0) + \epsilon_{s_o} - x]^2
\end{equation*}
corresponds to finding the metric projection of $\HV(K_0) + \epsilon_{s_o}$ on the nonempty, convex and closed set $H$. This projection always exists and is unique, see \cite[Sec. 1.2]{Schneider14}. Then the existence of a measure $\mu_{s_o} \in \M_{s_o}$ that satisfies \eqref{Inf_D} follows from Lemma~\ref{lemma_measure_existence}. The second statement of the lemma follows from the uniqueness of the projection. 
\end{proof}


Due to Lemma~\ref{lemma_infD} and the structure of $\M_{m_{s_o}}$, the minimization of $D_{s_o}$ can be reduced to the finite minimization problem 
\begin{equation}\label{mini_noise}
\inf_{(\alpha, \textbf{u}) \in \Ms} \sum_{s=0}^{s_o} \sum_{j=1}^{N(n,s)}\big( \psi_{(n-1)sj}(K_0) + \epsilon_{sj} - \sum_{l=1}^{m_{s_o}} \alpha_l H_{nsj}(u_l) \big)^2,
\end{equation}
where $\Ms$ is defined in \eqref{def_Ms}. A solution $(\alpha,\textbf{u}) \in \Ms$ to the minimization problem~\eqref{mini_noise} corresponds to the measure $\mu_{\alpha,\textbf{{u}}}=\sum_{j=0}^{m_{s_o}} \alpha_j \delta_{u_j} \in \M_{m_{s_o}}$. It follows from Lemma~\ref{lemma_supportmeasures} that the measure $\mu_{\alpha,\textbf{{u}}}$ is a surface area measure of a convex body in $\K^n$ if and only if $\mu_{\alpha,\textbf{{u}}}$ is of the form $a(\delta_v + \delta_{-v})$ for some $a \geq 0$ and $v \in S^{n-1}$ or if the matrix $M(\mu_{\alpha,\textbf{u}})$ of second order moments of $\mu_{\alpha,\textbf{u}}$ is positive definite. The assumption on $M(\mu_{\alpha,\textbf{u}})$ can alternatively be replaced by the assumption that $\alpha_1 u_1, \dots, \alpha_{m_{s_o}}u_{m_{s_o}}$ span $\R^n$ . 

Assume that $\mu_{\alpha,\textbf{{u}}}= a(\delta_v + \delta_{-v})$ for some $v \in S^{n-1}$ and $a \geq 0$. If $a=0$, we let $\hat{K}_{s_o}^H$ be the singleton $\{0\}$. If $a > 0$, we let $\hat{K}_{s_o}^H$ be a polytope in $u^{\perp}$ with surface area $a$. Now assume that $\alpha_1 u_1, \dots, \alpha_{m_{s_o}}u_{m_{s_o}}$ span $\R^n$. Then $\mu_{\alpha, \textbf{u}}$ is the surface area measure of a polytope with nonempty interior. We let $\hat{K}_{s_o}^H$ be the output polytope from Algorithm MinkData (see \cite[Sec. A.4]{Gardner06}) that reconstructs a polytope with surface area measure $\mu_{\alpha, \textbf{u}}$ from $(\alpha, \textbf{u})$. In all three cases, the surface area measure of  $\hat{K}_{s_o}^H$ is $\mu_{\alpha,\textbf{{u}}}$, so $\hat{K}_{s_o}^H$ minimizes $D_{s_o}$. 

As $s_o \geq 2$, it follows from  Lemma~\ref{lemma_supportmeasures} and the uniqueness statement of Lemma~\ref{lemma_infD} that if $\mu_{\alpha,\textbf{u}}$ is not a surface area measure of a convex body, then the same holds for every measure in $\M_0$ that minimizes $D_{s_o}$. Hence, the mapping $D_{s_o}$ does not attain its infimum on $\K^n$, and there does not exist a convex body with harmonic intrinsic volumes that fit the measurements $\HV(K_0) + \epsilon_{s_o}$ in a least squares sense. In this case, the reconstruction algorithm does not have an output. By Lemma~\ref{Lemma_D_bound} in Section~\ref{subsec_HIVcons}, this situation only occurs when the measurements are too noisy. The reconstruction algorithm is summarized in the following.

\paragraph{Algorithm Harmonic Intrinsic Volume LSQ ($n$-dim)}
\begin{description}
\item[Input:] Measurements $\HV(K_0) + \epsilon_{s_o}$ of the harmonic intrinsic volumes up to degree $s_o \geq 2$ of an unknown convex body $K_0 \in \K_n^n$.

\item[Task:] Construct a polytope $\hat{K}_{s_o}^H$ with at most $m_{s_o}$ facets such that the harmonic intrinsic volumes up to order $s_o$ of $\hat{K}_{s_o}^H$ fit the measurements $\HV(K_0) + \epsilon_{s_o}$ in a least squares sense.
 
\item[Action:]

Let $(a, \textbf{v})$ be a solution to the minimization problem 
\begin{equation*}
\inf_{(\alpha, \textbf{u}) \in \Ms} \sum_{s=0}^{s_o} \sum_{j=1}^{N(n,s)}\big( \psi_{(n-1)sj}(K_0) + \epsilon_{sj} - \sum_{l=1}^{m_{s_o}} \alpha_l H_{nsj}(u_l) \big)^2.
\end{equation*}

\begin{description}
\item[Case 1:] If $a=\textbf{0}$, let $\hat{K}_{s_o}^H=\{0\}$.
\item[Case 2:] If $\mu_{a, \textbf{v}}=\alpha(\delta_u + \delta_{-u})$ for some $\alpha > 0$ and $u \in S^{n-1}$, let $\hat{K}^H_{s_o}$ be a polytope in $u^{\perp}$ with surface area $\alpha$.
\item[Case 3:] If $a_1v_1, \dots, a_{m_{s_o}}v_{m_{s_o}}$ span $\R^n$, then $(a, \textbf{v})$ corresponds to the surface area measure of polytope $P \in \K^n_n$. Use Algorithm MinkData to reconstruct $P$, and let $\hat{K}_{s_o}^H=P$. 
\item[Case 4:] Otherwise, the solution $(a, \textbf{v})$ does not correspond to a surface area measure of a convex body. The algorithm has no output. 

\end{description}

\end{description}

\subsection{Consistency of the reconstruction algorithm} \label{subsec_HIVcons}
Let $(\Omega, \F, \Prob)$ be a complete probability space where the vectors of noise variables $(\epsilon_{s_o})_{s_o \geq 2}$ are defined. We assume that the noise variables are independent with zero mean and that the variance of $\epsilon_{s_oj}$ is bounded by $\sigma_{s_o}^2 > 0$ for $s_o \geq 2$ and $j=1, \dots, m_{s_o}$. In the following, for $s_o \geq 2$, we write 
$$D_{s_o}(\cdot, \epsilon_{s_o})=\norm[\HV(K_0) + \epsilon_{s_o} - \HV(\cdot)]^2$$ 
to emphasize the dependence of $D_{s_o}$ on $\epsilon_{s_o}$, and we let $r_{K_0}=\frac{r}{2}$ and $R_{K_0}=2R$, where $r$ and $R$ are defined as in \eqref{rR} with $K$ replaced by $K_0$. 

\begin{lemma}\label{Lemma_D_bound}
There exists a constant $c_{K_0} > 0$ such that any measure $\mu \in \M_0$ that minimizes $D_{s_o}(\cdot, \epsilon_{s_o})$ is the surface area measure of a convex body $K_{\mu} \in \K^n(r_{K_0}, R_{K_0})$ if $\norm[\epsilon_{s_o}] < c_{K_0}$.
\end{lemma}

\begin{proof}
If $\mu \in \M_0$ minimizes $D_{s_o}(\cdot, \epsilon_{s_o})$, then
\begin{align*}
\norm[{\HV[2](K_0) - \HV[2](\mu)}] &\leq \norm[\HV(K_0) + \epsilon_{s_o} - \HV(\mu)] + \norm[\epsilon_{s_o}]
\\
&\leq \sqrt{D_{s_o}(K_0, \epsilon_{s_o})} + \norm[\epsilon_{s_o}] = 2 \norm[\epsilon_{s_o}].
\end{align*}
The second order moments of $\mu$ depend linearly on $\HV[2](\mu)$, and the eigenvalues of the matrix of second order moments $M(\mu)$ of $\mu$ depend continuously on $M(\mu)$, see \cite[Prop. 6.2]{Serre2002}, so for each $\alpha > 0$,   
\begin{equation}\label{min_eigenvalue}
\abs[\lambda_{min}(M(S_{n-1}(K_0, \cdot))) - \lambda_{min}(M(\mu))] < \alpha
\end{equation}
if $\norm[\epsilon_{s_o}]$ is sufficiently small. Here $\lambda_{min}(A)$ denotes the smallest eigenvalue of a symmetric matrix $A$. Due to Lemma~\ref{lemma_supportmeasures} \eqref{item1support}, we have $\lambda_{min}(M(S_{n-1}(K_0, \cdot))) > 0$ as $K_0$ has nonempty interior, so $M(\mu)$ is positive definite if $\norm[\epsilon_{s_o}]$ is sufficiently small. Then $\mu$ is a surface area measure of a convex body $K_{\mu} \in \K^n_n$ by Lemma~\ref{lemma_supportmeasures}. Due to translation invariance of $K \mapsto S_{n-1}(K, \cdot)$, we can choose $K_{\mu}$ with centre of mass at the origin. Then by Lemma~\ref{spheres}, \eqref{min_eigenvalue} and the fact that 
$$\abs[S(K_0) -S(K_{\mu})] = \sqrt{\omega_n}\norm[{\HV[0](K_0) - \HV[0](\mu)}] \leq 2 \sqrt{\omega_n}\norm[\epsilon_{s_o}],$$
we even have that $r_{K_0}B^n \subseteq K_{\mu} \subseteq R_{K_0}B^n$ if $\norm[\epsilon_{s_o}] < c_{K_0}$, where $c_{K_0}> 0 $ is chosen sufficiently small. 
\end{proof}

We let $\Ks$ be the random set of convex bodies that minimize $D_{s_o}( \cdot, \epsilon_{s_o})$, i.e.
\begin{equation*}
\Ks=\big\{K \in \K^n \mid D_{s_o}(K, \epsilon_{s_o})=\inf_{L \in \K^n}D_{s_o}(L, \epsilon_{s_o}) \big\}.
\end{equation*}
By Equation~\eqref{infinf}, the set $\Ks$ is nonempty if and only if Algorithm Harmonic Intrinsic Volume LSQ has an output. Let $g \colon \K^n \times \R^{m_{s_o}} \to \R$ be given as $g(K,x)=\inf_{L \in \K^n} D_{s_o}(L, x) - D_{s_o}(K,x)$ for $K \in \K^n$ and $x \in \R^{m_{s_o}}$, then
\begin{equation*}
\bigg\{\Ks \neq \emptyset\bigg\}= \bigg\{\sup_{K \in \K^n} \1_{\{0\}}(g(K,\epsilon_{s_o})) = 1\bigg\} \subseteq \Omega,
\end{equation*}
and for $\alpha \in \R$, we have
\begin{align*}
&\bigg\{\sup_{K \in \Ks } \delta^t(K_0,K) \leq \alpha \bigg\} 
\\
=&\bigg\{\sup_{K \in \K^n} \delta^t(K_0,K)\1_{\{0\}}(g(K,\epsilon_{s_o})) \leq \alpha \bigg\} \cap \bigg\{\sup_{K \in \K^n} \1_{ \{0\}}(g(K,\epsilon_{s_o})) = 1\bigg\},
\end{align*}
where the supremum over the empty set is defined to be $\infty$. Using the notation of permissible sets, see \cite[App. C]{Pollard1984} and arguments as in \cite[p. 27]{Kousholt2016}, we obtain that $\sup_{K \in \K^n} \delta^t(K_0,K)\1_{\{0\}}(g(K,\epsilon_{s_o}))$ and $\sup_{K \in \K^n} \1_{ \{0\}}(g(K,\epsilon_{s_o}))$ are $\F$-$\B(\R)$-measurable. Then
\begin{equation*}
\bigg\{\sup_{K \in \Ks } \delta^t(K_0,K) \leq \alpha \bigg\} \in \F
\end{equation*}
for $\alpha \in \R$, which implies that $\sup_{K \in \Ks } \delta^t(K_0,K)$ is measurable.

\begin{theorem}\label{thm_cons_as}
Assume that $\sigma_{s_o}^2 = \mathcal{O}(s_o^{-(2n-1+\varepsilon)})$ for some $\varepsilon > 0$. Then
\begin{equation*}
\sup_{K \in \Ks} \delta^t (K_0,K) \to 0
\end{equation*}
almost surely for $s_o \to \infty$.
\end{theorem}

\begin{proof} 
It follows from the assumption on $\sigma_{s_o}^2$ that $m_{s_o}\norm[\epsilon_{s_o}]^2 \to 0$ almost surely for $s_o \to \infty$ as
\begin{equation*}
\sum_{s_o=2}^{\infty} \mathbb{E} m_{s_o}\norm[\epsilon_{s_o}]^2
= \sum_{s_o=2}^{\infty} m_{s_o} \sum_{j=1}^{m_{s_o}} \mathbb{E}\epsilon_{s_oj}^2
\leq \sum_{s_o=2}^{\infty} m_{s_o}^{2} \sigma_{s_o}^2 < \infty,
\end{equation*}
where we have used that $m_{s_o} = \mathcal{O}(s_o^{n-1})$ to obtain the last inequality. Now choose $c_{K_0}$ according to Lemma~\ref{Lemma_D_bound} and let $\omega \in \Omega$ satisfy that $m_{s_o}\norm[\epsilon_{s_o}(\omega)]^2 \to 0$ for $s_o \to \infty$.  Then, there exists an $S \in \N $ such that $\sqrt{m_{s_o}} \norm[\epsilon_{s_o}(\omega)] < c_{K_0}$ for $s_o > S$.  In particular, $\norm[\epsilon_{s_o}(\omega)] < c_{K_0}$ for $s_o > S$, so by Lemma~\ref{lemma_infD} and Lemma~\ref{Lemma_D_bound} there is an output polytope of Algorithm Harmonic Intrinsic Volume LSQ. Then, for $s_o > S$, the set $\mathbb{K}_{s_o}(\epsilon_{s_o}(\omega))$ is nonempty, and $K + x_K \in \K^n(r_{K_0}, R_{K_0})$ for $K \in \mathbb{K}_{s_o}(\epsilon_{s_o}(\omega))$ and a suitable $x_K \in \R^n$. Since 
\begin{align*}
\norm[\HV(K_0) - \HV(K)]
&\leq 
\norm[\HV(K_0) + \epsilon_{s_o}(\omega)- \HV(K)] + \norm[\epsilon_{s_o}(\omega)]
\\
&\leq
\sqrt{D_{s_o}(K_0, \epsilon_{s_o}(\omega))} + \norm[\epsilon_{s_o}(\omega)] =  2\norm[\epsilon_{s_o}(\omega)]
\end{align*}
for $K \in \mathbb{K}_{s_o}(\epsilon_{s_o}(\omega))$, the translation invariance of $K \mapsto S_{n-1}(K, \cdot)$ and Theorem~\ref{Thm_noise} yield that
\begin{align*}
&\sup_{K \in \mathbb{K}_{s_o}(\epsilon_{s_o}(\omega))} d_{D}(S_{n-1}(K_0, \cdot), S_{n-1}(K, \cdot))
\\
&\leq c(n,R_{K_0}, \frac{1}{3})s_o^{-\frac{1}{3}} + 2 \sqrt{\omega_n m_{s_o}} \norm[\epsilon_{s_o}(\omega)]
\to 0
\end{align*}
for $s_o \to \infty$. Hence, \cite[Lemma 9.5]{Gardner2006} and \cite[Thm. 8.5.3]{Schneider14} imply that
\begin{equation*}
\sup_{K \in \mathbb{K}_{s_o}(\epsilon_{s_o}(\omega))} \delta^t
 (K_0,K) \to 0
\end{equation*}
for $s_o \to \infty$.
\end{proof}

\begin{theorem}\label{thm_cons_prob}
Assume that $\sigma_{s_o}^2 = \mathcal{O}(s_o^{-(2n-2+\varepsilon)})$ for some $\varepsilon > 0$. Then 
\begin{equation*}
\sup_{K \in \mathbb{K}_{s_o}(\epsilon_{s_o})} \delta^t (K_0, K) \to 0
\end{equation*}
in probability for $s_o \to \infty$.
\end{theorem}

Markov's inequality and the assumption that $\sigma_{s_o}^2 = \mathcal{O}(s_o^{-(2n-2+\varepsilon)})$ imply that $ m_{s_o}\norm[\epsilon_{s_o}]^2 \to 0$ in probability for $s_o \to \infty$. Then, Theorem~\ref{thm_cons_prob} follows in the same way as Theorem~\ref{thm_cons_as}.

%
%

Theorems~\ref{thm_cons_as} and \ref{thm_cons_prob} yield that the reconstruction algorithm gives good approximations to the shape of $K_0$ for large $s_o$ under certain assumptions on the variance of the noise variables. To test how noise affects the reconstructions for small $s_o$, the ellipsoid in Figure~\ref{Ellipsoid2} is reconstructed from harmonic intrinsic volumes up to degree $6$. For $k \in \N_0$, the dimension of $\Haus^3_k$ is $2k+1$, and to derive the harmonic intrinsic volumes, we use the orthonormal basis of $\Haus_k^3$ given by
\begin{equation*}
H_{3k(2j+1)}(u(\theta, \phi))=\alpha_{kj}\sin^j(\theta)C_{k-j}^{j+\frac{1}{2}}(\cos(\theta))\cos(j\phi) , \qquad 0 \leq j \leq k
\end{equation*}
and
\begin{equation*}
H_{3k(2j)}(u(\theta, \phi))=\alpha_{kj}\sin^j(\theta)C_{k-j}^{j+\frac{1}{2}}(\cos(\theta))\sin(j\phi), \qquad 1 \leq j \leq k, 
\end{equation*}
where $\alpha_{kj} \in \R$ is a normalizing constant, $C_l^\lambda, l \in \N_0, \lambda > 0$ are Gegenbauer polynomials and $u(\theta, \phi)=(\sin(\theta)\sin(\phi), \sin(\theta)\cos(\phi), \cos(\theta))$ for $0 \leq \theta \leq \pi$ and $0 \leq \phi \leq 2\pi$, see \cite[Sections 1.2 and 1.6.2]{Dai2013}. 

The harmonic intrinsic volumes are subject to an increasing level of noise. The first plot in Figure~\ref{NoiseRecon} is a reconstruction based on exact harmonic intrinsic volumes, whereas the reconstructions in the second and third plot are based on harmonic intrinsic volumes disrupted by noise. The variance of the noise variables is $\sigma_2^2=1$ in the second plot and $\sigma_3^2=4$ in the third plot. Then the standard deviations $\sigma_2$ and $\sigma_3$ of the noise variables are approximately $5\%$ and $10\%$ of $\psi_{201}(K_0)$, respectively. For the three levels of noise, the minimization problem~\eqref{mini_noise} is solved using the \emph{fmincon} procedure provided by MatLab and Algorithm MinkData is applied to reconstruct a polytope corresponding to the solution.

\begin{figure}
\centering
\includegraphics[width=4.8cm]{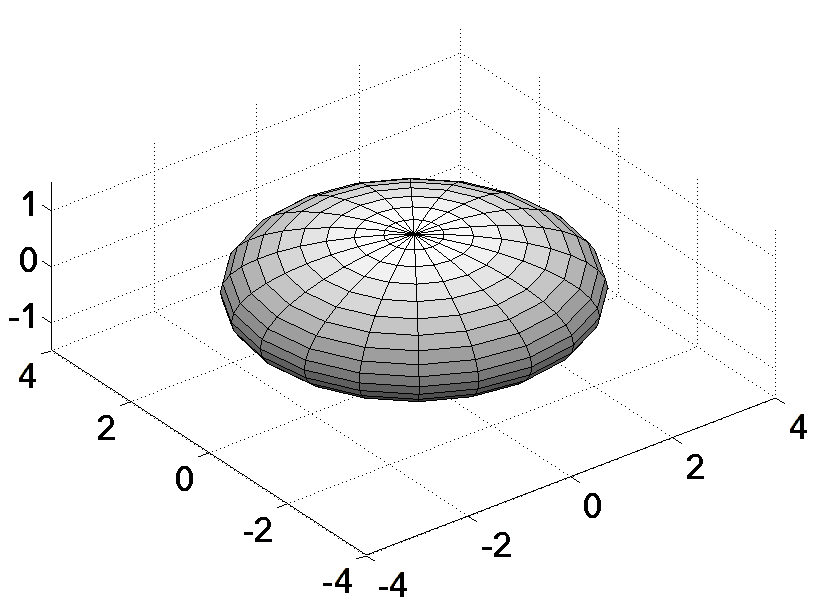}
\caption{Ellipsoid}
\label{Ellipsoid2}
\end{figure}

\begin{figure}
\centering
\includegraphics[width=4.8cm]{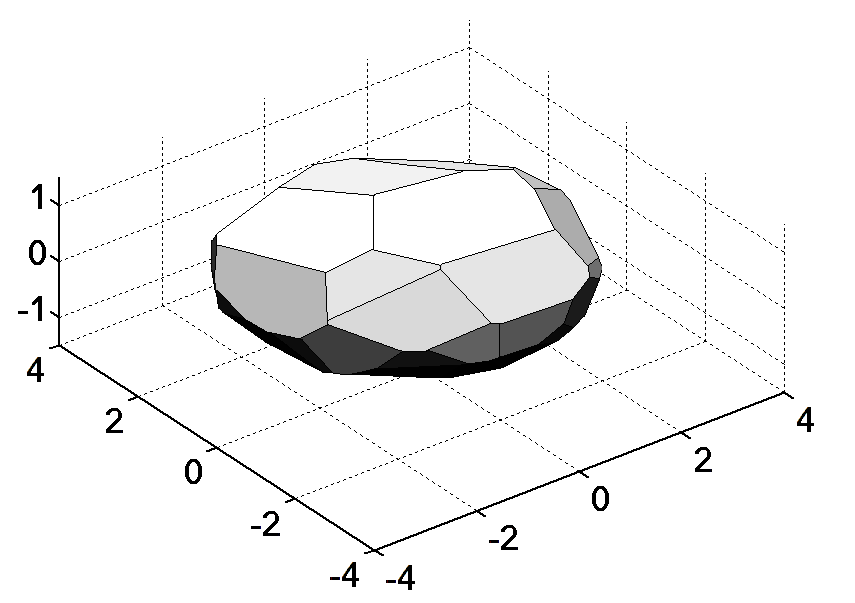}
\hspace*{0.7cm}
\includegraphics[width=4.8cm]{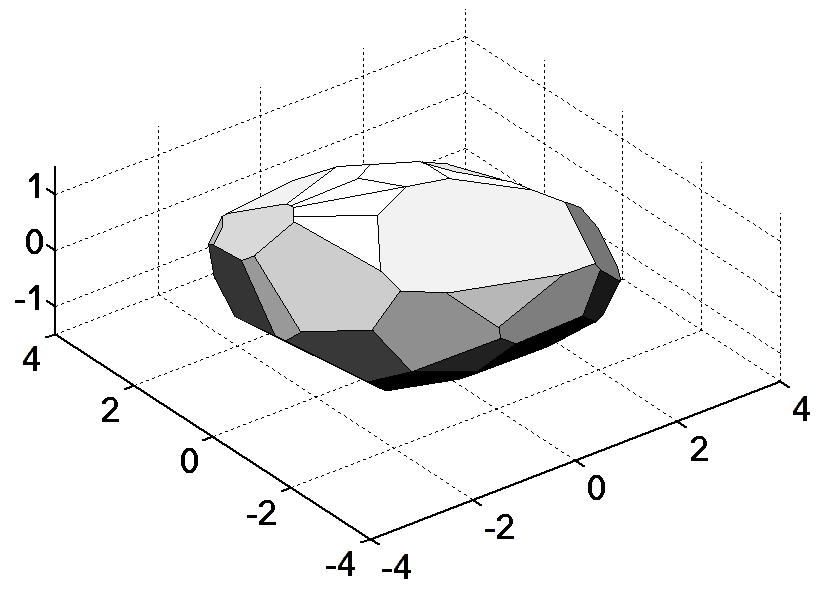}
\hspace*{0.7cm}
\includegraphics[width=4.8cm]{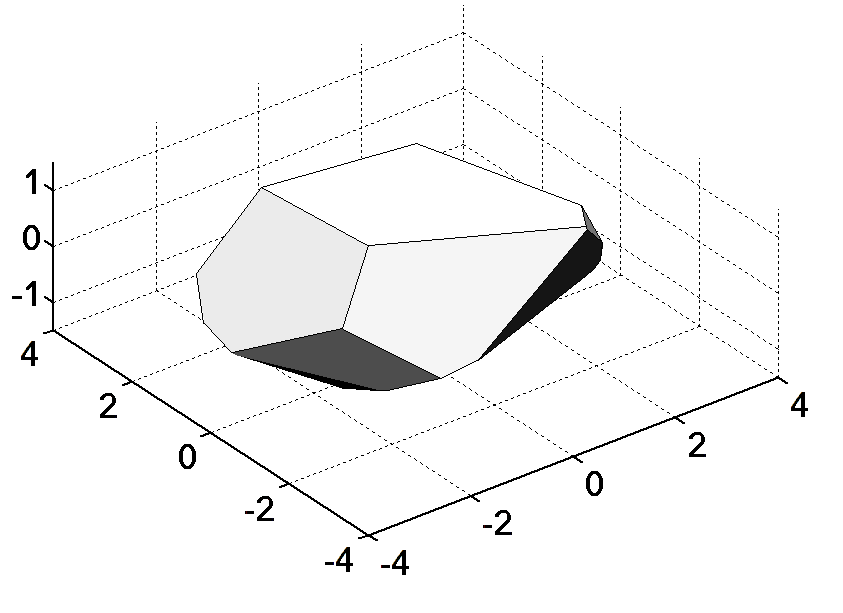}
\caption{Reconstructions of the ellipsoid in Figure~\ref{Ellipsoid2} based on noisy measurements of harmonic intrinsic volumes up to degree $s_o=6$. In the three plots, the variances of the noise variables are $0, 1$ and $4$.}
\label{NoiseRecon}
\end{figure}

The three plots in Figure~\ref{NoiseRecon} show how the reconstructions deviate increasingly from the ellipsoid
as the variance of the noise variables increases. The reconstruction based on exact harmonic intrinsic volumes captures essential features of the ellipsoid. The reconstruction is approximately invariant under rotations around the third axis and has the same main directions and semi axes lengths as the ellipsoid. Despite a noise level corresponding to $5\%$ of $\psi_{201}(K_0)$, the reconstruction in the second plot captures to some extent the same features and provides a fairly good approximation of the ellipsoid. The reconstruction in the third plot is comparable to the ellipsoid. However, the effect of noise is clearly visible.

\section*{Acknowledgements}
I thank Richard Gardner and Peyman Milanfar for making their implementation of the algorithm MinkData available to me. I am very grateful to Markus Kiderlen for his many helpful ideas and suggestions and for his comments on a first version of this paper. This research was supported by Centre for Stochastic Geometry and Advanced Bioimaging, funded by a grant from the Villum foundation.

\bibliography{arXiv}

\begin{thebibliography}{10}

\bibitem{Atkinson2012}
K.~Atkinson and W.~Han.
\newblock {\em Spherical Harmonics and Approximations on the Unit Sphere: An
  Introduction}.
\newblock Lecture Notes in Mathematics. Springer, Berlin, 2012.

\bibitem{Bauer2001}
H.~Bauer.
\newblock {\em Measure and integration theory}.
\newblock De Gruyter, Berlin, 2001.

\bibitem{Beisbart2002}
C.~Beisbart, R.~Dahlke, K.~Mecke, and H.~Wagner.
\newblock Vector- and tensor-valued descriptors for spatial patterns.
\newblock In K.~Mecke and D.~Stoyan, editors, {\em Morphology of Condensed
  Matter}. Springer, Heidelberg, 2002.

\bibitem{Dai2013}
F.~Dai and Y.~Xu.
\newblock {\em Approximation Theory and Harmonic Analysis on Spheres and
  Balls}.
\newblock Springer, New York, 2013.

\bibitem{Dudley2002}
R.~M. Dudley.
\newblock {\em Real Analysis and Probability}.
\newblock Cambridge University Press, Cambridge, 2002.

\bibitem{Gardner06}
R.~J. Gardner.
\newblock {\em Geometric Tomography}.
\newblock Cambridge University Press, New York, second edition, 2006.

\bibitem{Gardner2006}
R.~J. Gardner, M.~Kiderlen, and P.~Milanfar.
\newblock Convergence of algorithms for reconstructing convex bodies and
  directional measures.
\newblock {\em Ann. Stat.}, 34(3):1331--1374, 2006.

\bibitem{Groemer1996}
H.~Groemer.
\newblock {\em Geometric applications of Fourier series and spherical
  harmonics}.
\newblock Cambridge University Press, Cambridge, 1996.

\bibitem{Kousholt2016}
A.~Kousholt and M.~Kiderlen.
\newblock Reconstruction of convex bodies from surface tensors.
\newblock {\em Adv. Appl. Math.}, 76:1--33, 2016.

\bibitem{Lemordant1993}
J.~Lemordant, P.~D. Tao, and H.~Zouaki.
\newblock Mod\'{e}lisation et optimisation num\'{e}rique pour la reconstruction
  d'un poly\`{e}dre \`{a} partir de son image gaussienne
  g\'{e}n\'{e}ralis\'{e}e.
\newblock {\em RAIRO, Mod\'{e}lisation Math. Anal. Num\'{e}r.}, 27:349--74,
  1993.

\bibitem{Pollard1984}
D.~Pollard.
\newblock {\em Convergence of stochastic processes}.
\newblock Springer-Verlag, New York, 1984.

\bibitem{Schneider14}
R.~Schneider.
\newblock {\em Convex Bodies: The Brunn-Minkowski Theory}.
\newblock Cambridge University Press, Cambridge, second edition, 2014.

\bibitem{SM10}
G.~E. Schr{\"o}der-Turk, S.~Kapfer, B.~Breidenbach, C.~Beisbart, and K.~Mecke.
\newblock Tensorial {M}inkowski functionals and anisotropy measures for planar
  patterns.
\newblock {\em J. Microsc.}, 238(1):57--74, 2010.

\bibitem{Schroder-Turk2013}
G.~E. Schr{\"o}der-Turk, W.~Mickel, S.~C. Kapfer, F.~M. Schaller,
  B.~Breidenbach, D.~Hug, and K.~Mecke.
\newblock Minkowski tensors of anisotropic spatial structure.
\newblock {\em New J. Phys.}, 15:083028, 2013.

\bibitem{Serre2002}
Denis Serre.
\newblock {\em Matrices: Theory and Applications}.
\newblock Springer-Verlag, New York, 2002.

\end{thebibliography}

\end{document}